\documentclass[lettersize,journal]{IEEEtran}
\usepackage{amsmath,amsthm,amsfonts}
\usepackage{algorithmic}
\usepackage{array}
\usepackage[caption=false,font=normalsize,labelfont=sf,textfont=sf]{subfig}
\newtheorem{definition}{Definition}
\newtheorem{lem}{Lemma}
\newtheorem{theorem}{Theorem}
\newtheorem{remark}{Remark}
\usepackage{textcomp}
\usepackage{stfloats}
\usepackage{url}
\usepackage{verbatim}
 \usepackage{cite}
\usepackage{graphicx}
\def\BibTeX{{\rm B\kern-.05em{\sc i\kern-.025em b}\kern-.08em
    T\kern-.1667em\lower.7ex\hbox{E}\kern-.125emX}}
\usepackage{balance}
\begin{document}
\title{Privacy-Preserving Formation Control for Networked Underactuated USVs: A Passivity-Based Approach}
\author{Jingyi Zhao, Wenxuan Wang, Weijun Zhou, Yongxin Wu, Yuhu Wu, Yann Le Gorrec
\thanks{J. Y. Zhao and Y. H. Wu (Corresponding author) are with the Key Laboratory of Intelligent Control and Optimization for Industrial Equipment of Ministry of Education and the School of Control Science and Engineering, Dalian University of Technology, Dalian, 116024, China. (E-mail: zhaojingyi@dlut.edu.cn; wuyuhu@dlut.edu.cn)}
\thanks{W. X. Wang is with Zhejiang University, Hangzhou, 310015, China. (E-mail: wenxuan\_wang@zju.edu.cn)}
\thanks{Y. X. Wu  and Y. Le Gorrec are with Universit\'e Marie et Louis Pasteur, SUPMICROTECH, CNRS, institute FEMTO-ST, F-25000, Besan\c{c}on, France. (E-mail: yongxin.wu@femto-st.fr; yann.le.gorrec@ens2m.fr)}
\thanks{W. J. Zhou is with School of Information and Electrical Engineering, Hangzhou City University, Hangzhou, 310015, China. (E-mail: zhouweijun0086@gmail.com)
}
}

\markboth{IEEE/ASME Transactions on mechatronics,
	~Vol.~XX, No.~X, February~2026}%
{How to Use the IEEEtran \LaTeX \ Templates}

\maketitle

\begin{abstract}
This paper studies coordinated trajectory planning and tracking control for multiple unmanned surface vessels (USVs) under strict privacy requirements. To avoid the privacy risks associated with direct position sharing in conventional cooperative methods, the proposed approach adopts an estimated fleet centroid as the only shared variable, preventing individual trajectory disclosure while enabling coordination. Based on this interaction mechanism, a formation-oriented trajectory is generated for the fleet. The collective dynamics are modeled using Port-Hamiltonian systems, and a passivity-based tracking controller is designed for each USV to accurately follow the planned trajectories. The stability of the closed-loop system is rigorously proven, and experiments on a real USV platform confirm effective formation tracking and privacy preservation. The proposed result extends and validates through experimental results the approach in \cite{rf26} that was limited to idealized point-mass models and lacked a feedback control.
\end{abstract}


\section{Introduction}
\IEEEPARstart{W}{ith} the rapid advancement of maritime robotics, multiple unmanned surface vessels (USVs) have emerged as vital assets for various marine operations, including oceanographic mapping, environmental monitoring, and coordinated search-and-rescue \cite{rf4,rf5,rf33}. To successfully execute these missions, precise formation control can improve operational efficiency, as the fleet is required to achieve and maintain a predefined geometric configuration \cite{rf2,rf3}. In practice, as these systems transition from centralized to distributed architectures, inter-USV coordination becomes increasingly dependent on the reliability and privacy security of wireless communication networks \cite{rf6}.

A critical aspect of this security concern, often overlooked in conventional formation control \cite{rf7,rf8,rf9}, is the risk of privacy leakage in the communication process. Despite the benefits of collaborative autonomy, the frequent exchange of state information (e.g., real-time positions and velocities) exposes USVs to severe privacy risks \cite{rf10}. In privacy-sensitive scenarios, such as maritime monitoring or sensitive commercial surveying, direct state sharing can be exploited by internal Honest-But-Curious (HBC) adversaries or external eavesdroppers \cite{rf11}. These adversaries can reconstruct individual trajectories to infer mission intentions or identify vulnerabilities within the formation. Conventional privacy-preserving techniques, such as differential privacy, typically involve injecting additive noise into communicated signals \cite{rf12}. While effective, the introduced stochastic noise often causes steady-state offsets, preventing the fleet from achieving the high-precision alignment required for marine tasks \cite{rf13}. Alternatively, cryptographic solutions like homomorphic encryption \cite{rf14} offer high security but impose significant computational overhead and communication latency, which are often prohibitive for real-time control on embedded USV platforms.

Most existing privacy-preserving coordination schemes treat USVs as simplified mass-point models or single/double-integrators \cite{rf15,rf16,rf17,rf18}. Such oversimplification overlooks the complex nonlinear hydrodynamics, including underactuated constraints and time-varying environmental disturbances, such as wind and waves, which may degrade tracking performance or even cause instability during aggressive maneuvers. 
The port-Hamiltonian (PH) framework provides an effective modeling approach for USVs, as its physically intuitive representation explicitly describes the internal energy storage and external dissipation \cite{rf19,rf20}. 
Moreover, the inherent PH structure naturally induces passivity, facilitating passivity based control design, and providing a naturally stable, energy-consistent backbone for multi-USV coordination \cite{rf21,rf22}. 
While several research groups have successfully explored the formation controller within the PH framework, these studies generally focus on tracking performance without considering the inherent privacy risks during information exchange \cite{rf23,rf24,rf25}. 
This is primarily due to the conflict between privacy requirements and stability guarantees. 
Standard privacy-preserving techniques, such as noise injection or encryption, typically introduce stochastic disturbances or computational delays that may destroy the passivity of PH systems. Consequently, it is nontrivial to design a coordination law that ensures trajectory privacy while maintaining the energy-based stability of the closed-loop system. 
Although our preliminary study \cite{rf26} introduced a privacy-preserving trajectory planning scheme to address this challenge, it was limited to idealized point-mass models and lacked a feedback controller. 
Moreover, several technical challenges remain when transitioning from theoretical trajectory planning to real-time onboard execution. First, the idealized point-mass models adopted in previous studies cannot be directly extended to physical USVs, as they neglect the high-order nonlinear dynamics and underactuated constraints inherent in real-world systems. Second, environmental disturbances may compromise the passivity of the closed-loop system, potentially leading to formation instability. Addressing these challenges requires a unified framework that ensures both privacy preservation and dynamical robustness.

To this end, this paper develops an enhanced hierarchical control architecture for multi-USV systems. By synthesizing a centroid estimation-based privacy mechanism with a passivity-based tracking controller, the proposed scheme achieves precise formation maintenance while rendering trajectories indistinguishable under environmental perturbations. The main contributions are summarized in what follows:
\begin{itemize}
	\item Unlike existing PH-based formation studies \cite{rf23,rf24,rf25} that neglect privacy preservation, this paper develops a hierarchical control framework that integrates trajectory privacy with dynamical control. By synthesizing a centroid estimation-based privacy mechanism with a passivity-based tracking controller, the proposed framework ensures that individual trajectories (privacy) remain indistinguishable without compromising the physical-layer formation performance.
	\item In contrast to classical privacy-preserving methods that rely on oversimplified point-mass models \cite{rf17,rf30}, this paper develops a centroid estimation-based formation controller by explicitly incorporating underactuated USV dynamics. This approach avoids direct state exchange while preserving the passivity of the closed-loop system, thereby guaranteeing convergence to the exact desired formation without the steady-state offsets typically induced by stochastic noise injection in differential privacy methods \cite{rf12,rf28}. Furthermore, the proposed scheme circumvents the high computational overhead and time-delays associated with encryption-decryption \cite{rf18} or state decomposition processes \cite{rf29}, ensuring its suitability for real-time onboard execution.
	\item In contrast to existing studies such as \cite{rf26,rf27,rf32} that primarily focus on theoretical analysis or single-USV control, this paper presents a practical implementation of the proposed privacy-preserving formation controller for a multi-USV fleet. Experimental results validate the efficiency of the proposed framework using underactuated USV platforms in real-world disturbance environments.
\end{itemize}

The remainder of this paper is organized as follows. Section II provides the necessary preliminaries. The  formation control problem for underactuated USVs is given in Section III. Section IV presents the main results, and the efficiency of the proposed method is validated through physical experiments in Section V. Finally, Section VI provides some conclusions and perspectives.
\section{Preliminaries}
\subsection{Graph Theory}
The interaction topology among the $N$ USVs is represented by an undirected connected graph $\mathcal{G} = \{\mathcal{V}, \mathcal{E}, \mathcal{A}\}$. Here, $\mathcal{V} = \{1, \dots, N\}$ and $\mathcal{E} \subseteq \mathcal{V} \times \mathcal{V}$ denote the sets of nodes and edges, respectively. The communication weights are encoded in the symmetric adjacency matrix $\mathcal{A} = [a_{ij}] \in \mathbb{R}^{N \times N}$, where $a_{ij} = a_{ji} = 1$ if node $j$ is within the neighborhood of node $i$ (denoted by $j \in \mathcal{N}_i$), and $a_{ij} = 0$ otherwise. We assume no self-loops, i.e., $a_{ii} = 0$. Let $\mathcal{D} = \text{diag}\{d_1, \dots, d_N\}$ be the degree matrix with $d_i = \sum_{j \in \mathcal{N}_i} a_{ij}$. The Laplacian matrix is then formulated as $L = \mathcal{D} - \mathcal{A}$, which is positive semi-definite for undirected connected  graph $\mathcal{G}$. The eigenvalues of $L$ satisfy $0 = \lambda_1(L) < \lambda_2(L) \leq \dots \leq \lambda_N(L)$, where the algebraic connectivity is characterized by $\lambda_2(L) > 0$. 
For the sake of brevity, the index $i$ is implicitly assumed to belong to the node set $\mathcal{V}$ throughout the rest of this paper, unless otherwise specified.
\subsection{Port-Hamiltonian Systems}
Consider $N$ USVs whose communication topology is governed by the undirected connected graph $\mathcal{G}$. The dynamics of the $i$-th USV is given by:
\begin{equation}\label{gene}
	\dot{x}_i(t) = \left[J_i(x_i) - R_i(x_i)\right] \nabla H_i(x_i) + g_i(x_i) \tau_i(t),
\end{equation}
where $x_i(t) \in \mathbb{R}^m$ and $\tau_i(t) \in \mathbb{R}^n$ denote the system state and control input, respectively. The smooth Hamiltonian function $H_i(x_i): \mathbb{R}^m \to \mathbb{R}$ represents the total stored energy of the $i$-th USV, the interconnection matrix $J_i(x_i) \in \mathbb{R}^{m\times m}$ and the dissipation matrix $R_i(x_i)\in \mathbb{R}^{m\times m}$ satisfy $J_i(x_i) = -J_i^\top(x_i)$ and $R_i(x_i) = R_i^\top(x_i) \succeq 0$. The state-dependent input mapping is denoted by $g_i(x_i) \in \mathbb{R}^{m \times n}$. A comprehensive derivation of PH systems can be found in \cite{rf22}.
\subsection{USV Classifications and Adversaries}
We partition $\mathcal{V}$ into the neutral USV set $\mathcal{V} \setminus \mathcal{H}$ and HBC USV set $\mathcal{H}$. Their behaviors are categorized as follows \cite{rf17}:
\begin{itemize}
	\item Neutral USVs: Strictly follow the control protocols and maintain privacy neutrality; they neither attempt to infer others' states nor collude to conceal information.
	\item HBC USVs (internal adversaries): Correctly execute the designed dynamics but actively try to infer the privacy of other USVs by using the accessible information in the set $\mathcal{I}_h(t)$.
\end{itemize}

To model a privacy threat, we consider collusion and external eavesdropping:
\begin{itemize}
	\item HBC Collusion: Neighboring HBC adversaries $h, k \in \mathcal{H}$ will voluntarily share their entire local information sets to enhance inference. The accessible information set for a colluding HBC adversary $h \in \mathcal{H}$ is:
\begin{equation}\label{Eh1}
	\mathcal{I}^c_h(t) = \mathcal{I}_h(t) \cup\big( \cup_{k \in \mathcal{N}_h \cap \mathcal{H}} \mathcal{I}_k(t) \big).
\end{equation}
	If no collusion occurs, $\mathcal{I}^c_h(t) = \mathcal{I}_h(t)$.
	\item External eavesdropper: An eavesdropper possessing global topological knowledge $\mathcal{A}$ can wiretap all transmitted signals $\{F_i(t)\mid i \in \mathcal{V}\}$. Its accessible information set is:
	\begin{equation}\label{I}
		E^c(t) =\{ \mathcal{A} \} \cup \{F_i(t)\mid i \in \mathcal{V}\} \cup \{{ \mathcal{I}^c_h(t) \mid h \in \mathcal{H} }\},
	\end{equation}
	which simplifies to $E(t) = \{ \mathcal{A}, F_i(t) \mid i \in \mathcal{V} \}$ in the absence of colluding HBC adversaries.
\end{itemize}

\section{Problem Formulation}
\subsection{The USV model in PH form}
The kinematic and dynamic models of the $i$-th underactuated USV ($i\in\mathcal{V}$) related to surge, sway, and yaw motion are represented in Fig.~\ref{f1} and described as follows:
\begin{equation}\label{e1}
\!\!	\begin{cases}
		\dot q_i(t)=J(q_i)\nu_i(t),\\
		M_i \dot \nu_i(t)\!\!=\!\!-(C_i(\nu_i)+D_i(\nu_i))\nu_i(t)+\tau_{ci}(t)+\tau_{di}(t),
	\end{cases}
\end{equation}
where $q_i(t)=[x_i(t),y_i(t),\psi_i(t)]^\top$ consists of the positions $x_i(t)$, $y_i(t)$ in the earth frame along $x_E$, $y_E$, respectively, and the heading orientation $\psi_i(t)$. The vector $\nu_i(t)=[u_i(t),v_i(t),r_i(t)]^\top$ consists of the surge velocity $u_i(t)$, sway velocity $v_i(t)$ and the angular velocity $r_i(t)$, respectively. The generalized mass matrix  $M_i=\text{diag}(m_{1i},m_{2i},m_{3i})\in\mathbb{R}^{3\times 3}$ and the hydrodynamic damping matrix $D_i(\nu_i)=\text{diag}(d_{1i},d_{2i},d_{3i})\in\mathbb{R}^{3\times 3}$ are strictly positive definite, i.e.,  $M_i\succ 0$ and $D_i(\nu_i) \succ 0$ for all $\nu_i$. The rotation matrix $J(q_i)\in\mathbb{R}^{3\times 3}$ is denoted by
$$
J(q_i)=\begin{bmatrix}
	R(\psi_i) &0_{2\times 1}\\
	0_{1\times 2}& 1
\end{bmatrix},\quad R(\psi_i)=\begin{bmatrix}
\cos(\psi_i)&-\sin(\psi_i)\\ \sin(\psi_i)&\cos(\psi_i)
\end{bmatrix}.
$$
The Coriolis-centripetal matrix $C_i(v_i)$ related to the Coriolis force and centripetal force is denoted by
$$C_i(\nu_i)=\begin{bmatrix}
	0&0&-m_{2i}v_i(t)\\
	0&0&m_{1i}u_i(t)\\
	m_{2i}v_i(t)&-m_{1i}u_i(t)&0\end{bmatrix}.$$
The vector $\tau_{di}(t)\in\mathbb{R}^3$ denotes the bounded unknown disturbance of the $i$th USV, satisfying $\Vert \tau_{di}(t)\Vert\leq \tau_{max}$, and $\tau_{ci}(t)\in\mathbb{R}^3$ denotes the control input. 

\begin{figure}[!t]
	\centering
	\includegraphics[width=3in]{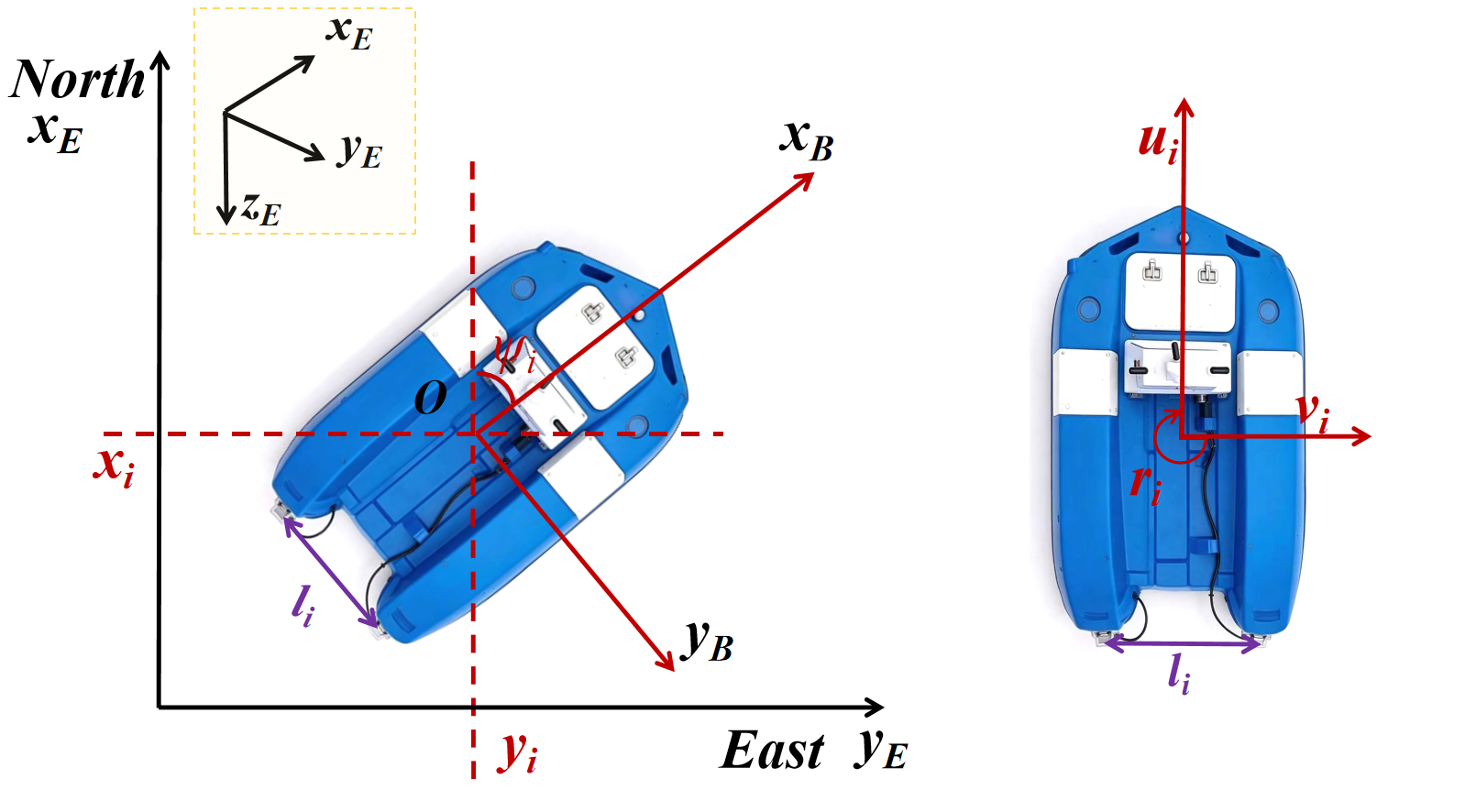}
	\caption{ Schematic depiction of an USV.}
	\label{f1}
\end{figure}
As shown in Fig.~\ref{f1}, each USV is driven by two electric motors in a symmetric configuration. Following the conventional modeling approach in \cite{fossen2011handbook}, the control inputs are transformed from individual thruster outputs into a total surge force and a resultant yaw torque as
\begin{equation}\label{e2}
	\tau_{ci}(t)=\begin{bmatrix}
		\tau_{ui}(t)\\	\tau_{vi}(t)\\	\tau_{ri}(t)
	\end{bmatrix}=\begin{bmatrix}
	F_{1i}(t)+F_{2i}(t)\\ 0\\ \frac{l_i(F_{1i}-F_{2i})}{2}(t)
	\end{bmatrix},
\end{equation}
where $F_{1i}, F_{2i}\in\mathbb{R}$ capture the propulsive forces delivered by the permanent magnet synchronous motors-driven propellers, and $l_i\in\mathbb{R}$ denotes the distance between two propellers.

For the $i$-th USV, let $p_i(t)\!=\! M_i \nu_i(t)$ denotes its momentum, where $\nu_i(t) =[u_i(t),v_i(t),r_i(t)]^\top$. The total energy is defined by the Hamiltonian function $H_i(p_i, q_i) = \frac{1}{2} p_i(t)^\top M_i^{-1} p_i(t) + V_i(q_i)$, in which $V_i(q_i)$ represents the potential energy associated with gravitational forces, defined to be zero at the water surface. Consequently, the port-Hamiltonian dynamics of the $i$-th USV can be formulated as follows \cite{Donaire2017}:
\begin{equation}\label{e3}
	\begin{aligned}
		\begin{bmatrix}
			\dot q_i(t)\\ \dot p_i(t)
		\end{bmatrix}&=\begin{bmatrix}
			0&J(q_i)\\-J^\top(q_i)&-\bar D_i(p_i)
		\end{bmatrix}\begin{bmatrix}
			\nabla_{q_i}H_i\\ \nabla_{p_i}H_i
		\end{bmatrix}
		\\&+\begin{bmatrix}
			0_{3\times 3}\\ I_3
		\end{bmatrix}(\tau_{ci}(t)+\tau_{di}(t)),
	\end{aligned}
\end{equation}
where $\bar D_i(p_i)=C_i(M_i^{-1}p_i)+D_i(M_i^{-1}p_i)$.
\subsection{Research objective}
The purpose of this work is to design a privacy-preserving formation controller for a fleet of USVs subject to unknown disturbances. The research objectives are formulated as follows:

(1) The displacement-based formation objective (Fig.~\ref{f2}-A):
\begin{equation}\label{e4}
	\lim_{t\rightarrow\infty} (Q_i(t)-Q_j(t)) = \delta^*_{ij},\quad i,j\in\mathcal{V},
\end{equation}
where $Q_i(t)=[x_i(t),y_i(t)]^\top\in\mathbb{R}^2$ denotes the position of the $i$-th USV,  $\delta^*_{ij}=Q_i^*-Q_j^*\in\mathbb{R}^2$ denotes the desired displacement between the $i$-th and the $j$-th USV ($i,j\in\mathcal{V}$), with the desired positions $Q_i^*$ and $Q_j^*$.
\begin{figure}[!t]
	\centering
	\includegraphics[width=3.5in]{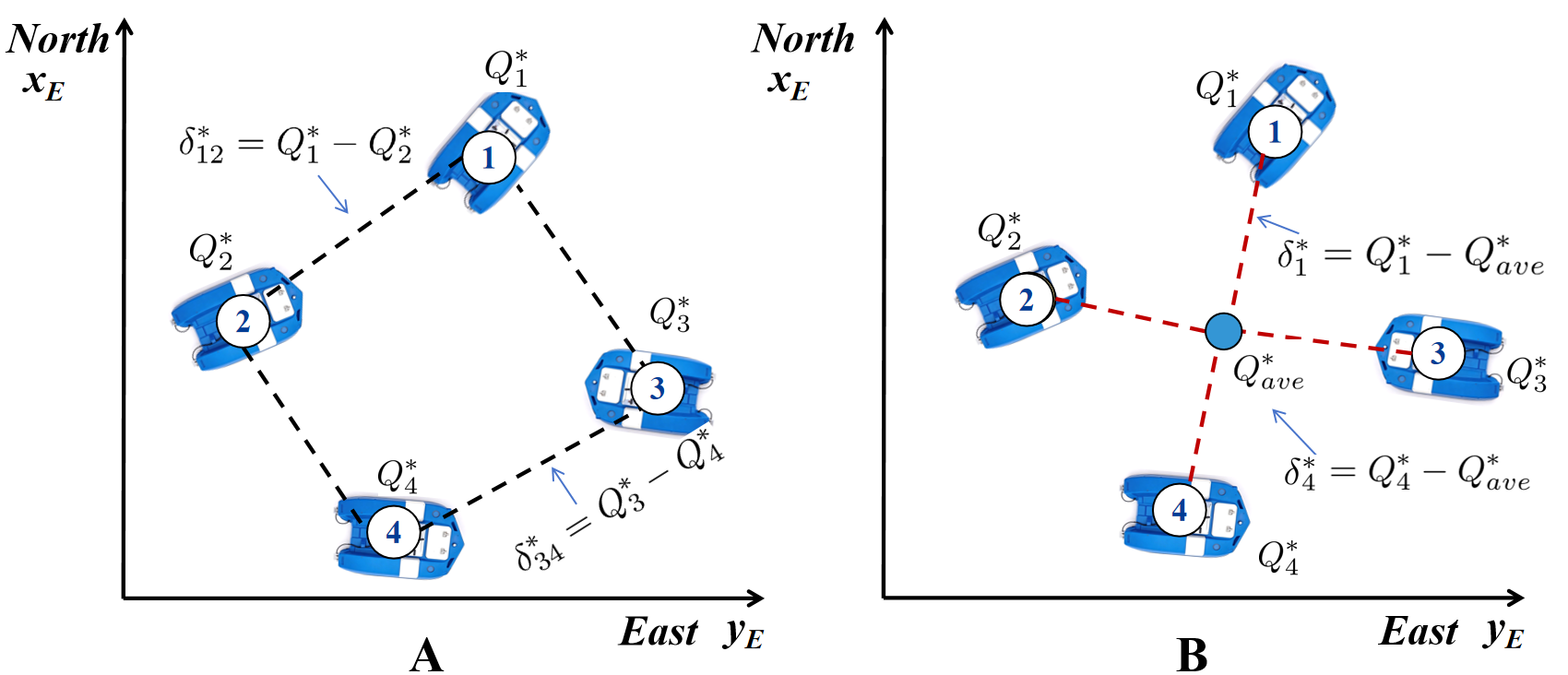}
	\caption{The displacement-based formation control objectives.}
	\label{f2} 
\end{figure}

(2) The privacy-preserving objective. Develop a defense mechanism that ensures the position trajectory $Q_i(t)$ (defined as privacy) of the $i$-th USV remains private against malicious estimation. Specifically, for both HBC adversaries and external eavesdroppers, the exact position trajectory $Q_i(t)$ should be made  indistinguishable from their observations, even if the interactive data is available to them.

\section{Main Results}
In this section, a hierarchical control framework is developed to achieve the privacy-preserving formation control objectives. As shown in Fig.~\ref{f4}, the strategy consists of two main components:
\begin{figure}[!t]
	\centering
	\includegraphics[width=3in]{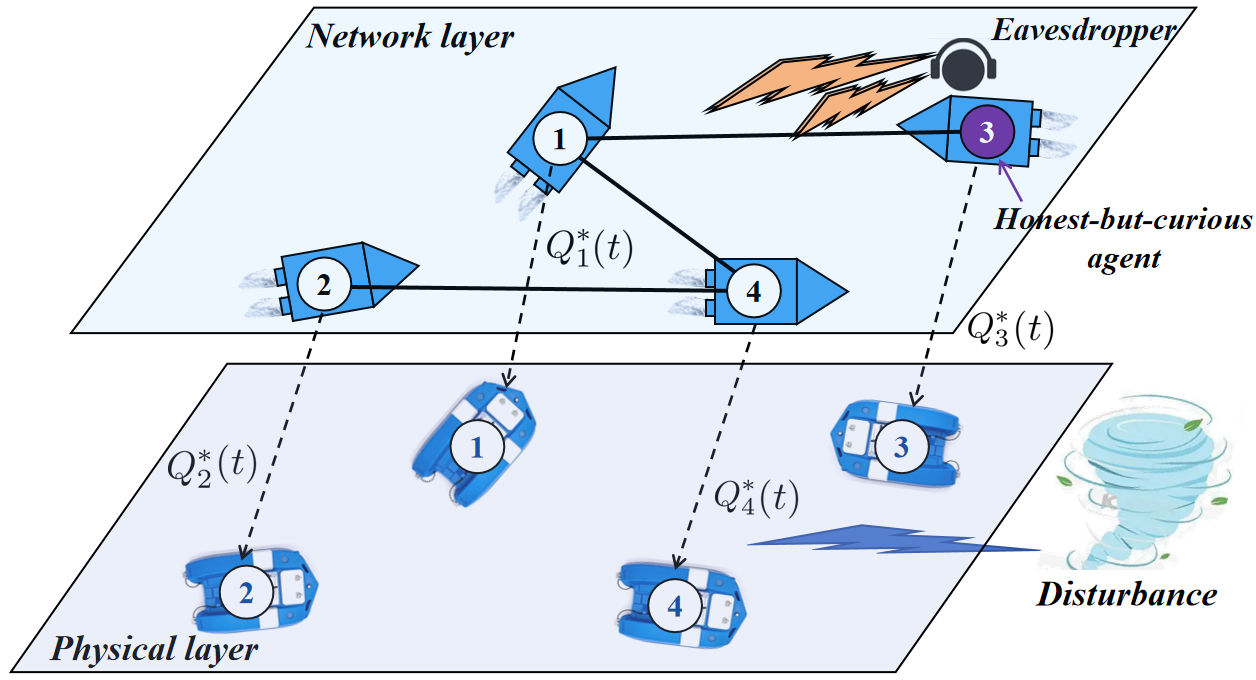}
	\caption{The scheme diagram of the hierarchical control framework.}
	\label{f4} 
\end{figure}
\begin{itemize}
	\item Privacy-preserving trajectory planning. To address privacy requirements, a reference trajectory $Q_i^*(t)$ is synthesized for the $i$-th USV by leveraging the consensus-based centroid estimation in the network layer. This mechanism ensures that the formation is achieved through the exchange of centroid  estimations rather than real-time positions, thereby embedding privacy preservation into the trajectory generation process.
	\item Dynamic tracking control. In the physical layer, a robust controller is synthesized to ensure that the actual USV dynamics track the planned kinematic trajectory in the presence of unknown disturbances.
\end{itemize}
\subsection{Privacy-preserving communication mechanism}\label{A}
To achieve the formation objective \eqref{e4}, a straightforward approach involves the exchange of the real-time position $Q_i(t)$, such as in \cite{rf24}. However, such direct interaction poses a significant risk of information leakage. To mitigate this, we introduce the desired formation's centroid $Q_{ave}^*$ as a virtual reference, as shown in Fig.~\ref{f2}-B. By maintaining a prescribed offset relative to this centroid, the fleet can maintain the desired shape without explicitly revealing individual coordinates. Hence, by defining $\delta_i^*=\frac{1}{N}\sum_{i=1}^N \delta_{ij}^*$, we can find that the objective \eqref{e4} is equivalent to 
\begin{equation}\label{e5}
	\lim_{t\rightarrow\infty} (Q_i(t)-Q_{ave}(t)) = \delta^*_{i},\quad i,j\in\mathcal{V},
\end{equation}
where $\delta^*_i=Q_i^*-Q_{ave}^*$ denotes the target displacement between the desired position $Q_i^*$ of the $i$th USV and the desired centroid $Q_{ave}^*=\frac{1}{N}\sum_{i=1}^N Q_i^*$ of the formation.

Since the real-time centroid $Q_{ave}(t)$ of USVs is unavailable to each USV, we propose a distributed consensus-based  estimator where each USV maintains a local belief $\hat{\eta}_i(t)$ of the fleet's centroid $Q_{ave}(t)$. Instead of exchanging the privacy $Q_i(t)$ with neighbors, USVs interact by sharing these centroid estimations $\hat\eta_i(t)$. The dynamics of $\hat{\eta}_i(t)$ are driven by the information exchange with neighbors, such that $\hat{\eta}_i(t) \to Q_{ave}(t)$ as $t \to \infty$. This mechanism decouples the required formation feedback from the private position trajectory, providing a fundamental layer for privacy preservation. The following example shows that the effectiveness of the proposed communication mechanism.

\subsubsection*{Example 1}
To demonstrate the privacy-preserving capability, consider a fleet of $4$-USVs indexed by $\mathcal{V}=\{1, 2, 3, 4\}$, organized in a undirected ring topology, as shown in Fig.~\ref{f3}. Under the proposed protocol, the transmitted information is the estimation  $\hat{\eta}_i(t)$. The $1$-st and the $2$-nd USV are neutral USVs, the $3$-th, $4$-th USVs are colluding HBC USVs, meaning they share all of their internal data (including their position $Q_3(t)$, $Q_4(t)$) to infer the private position $Q_1(t)$ and $Q_2(t)$. An eavesdropper intercepts all interactive messages transmitted over the communication links, including the estimation $\hat{\eta}_i(t)$ of all USVs and the position $Q_3(t)$, $Q_4(t)$ from the colluding HBC USVs. Under our mechanism, even with continuous monitoring of the estimation exchange, the observers face an ill-posed inverse problem. The real-time actual private trajectories $Q_1(t)$ (or $Q_2(t)$) are mapped to estimations $\hat{\eta}_1(t)$ (or $\hat{\eta}_2(t)$) via a privacy-preserving projection, which acts as a dynamic shield. This mapping ensures that multiple distinct trajectories can produce the same observable estimations, rendering the true trajectory indistinguishable to adversaries.
\begin{figure}[!t]
	\centering
	\includegraphics[width=3in]{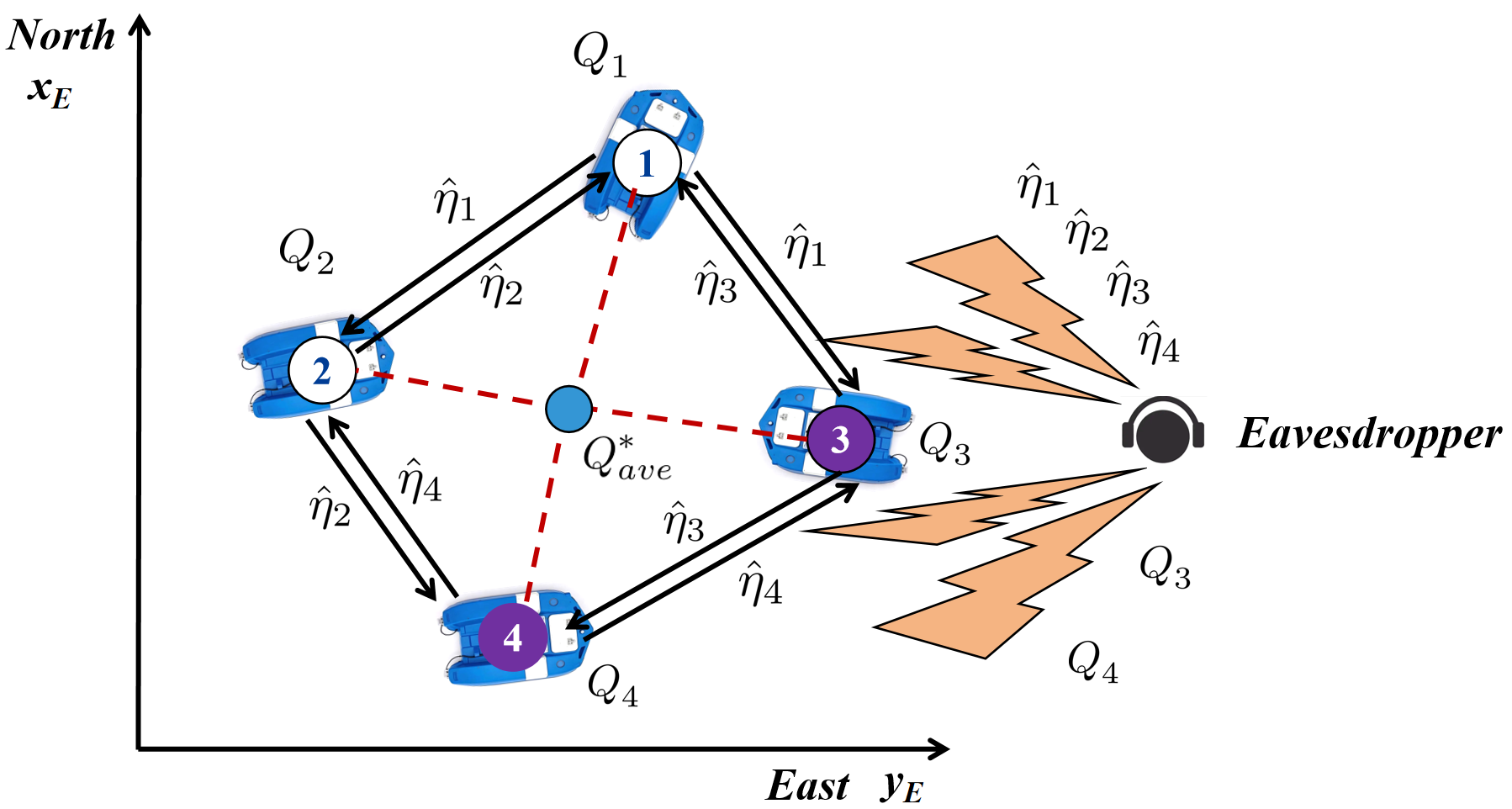}
	\caption{An example of the communication mechanism for $4$ USVs.}
	\label{f3} 
\end{figure}
\subsection{Kinematic trajectory planning}\label{S2}
Based on the privacy-preserving communication protocol defined in Section \ref{A}, we now synthesize a distributed trajectory planning algorithm for each USV. Firstly, we decompose the global formation task \eqref{e5} into $N$ coupled sub-optimization problems, where the $i$-th USV independently seeks its optimal trajectory based on local interactions.

The cost function of the $i$-th USV is designed as 
\begin{equation}\label{e6}
	V_i(Q_i,Q_{ave})=\frac{1}{2}\Vert Q_i(t)-Q_{ave}(t)-\delta_i^*\Vert^2.
\end{equation}
With the definition in \eqref{e6}, it is easy to verified that $V_i(Q_i,Q_{ave})$ is a continuously differentiable convex function with respect to $Q_i(t)$ if $Q_1(t),\cdots,Q_{i-1}(t),$ $Q_{i+1}(t),\cdots,Q_N(t)$ are fixed. Then by \cite{rf26}, the objective \eqref{e5} is reformulated into the following optimization problem:
\begin{equation}\label{p1}
	\min_{Q_i\in\mathbb{R}^2} V_i(Q_i,Q_{ave}), \quad \forall i\in\mathcal{V}.
\end{equation}
To solve the problem \eqref{p1}, the following lemma is given.
\begin{lem}\label{l1}\cite{rf34} 
The vector $Q^*=\text{col}(Q_1^*,\cdots,Q_N^*)\in\mathbb{R}^{2N}$ is one of the solutions of the optimization problem \eqref{p1} if and only if 
$$\nabla_{Q_i} V_i(Q_i,Q_{ave})\vert_{Q_i=Q_i^*} =0_2.$$
\end{lem}
Since the real-time centroid $Q_{ave}(t)$ is unaccessible to any USV, and the estimation $\hat\eta_i(t)$ is used for interaction, we define the following mappings for the $i$-th USV as
\begin{equation}\label{e8}
	\begin{aligned}
			\mathcal{C}_i(Q_i,\hat\eta_i)&=V_i(Q_i,Q_{ave})\vert_{\hat\eta_i=Q_{ave}}\\&=\frac{1}{2}\Vert Q_i(t)-\hat\eta_i(t)-\delta_i^*\Vert^2.
	\end{aligned}
\end{equation}
The gradient of $\mathcal{C}_i(Q_i,\hat\eta_i)$ with respect to $Q_i(t)$ and $\hat\eta_i(t)$ are defined by
$G_i(Q_i,\hat\eta_i):\mathbb{R}^2\times\mathbb{R}^2\rightarrow\mathbb{R}^2$ and $\Psi_i(Q_i,\hat\eta_i):\mathbb{R}^2\times\mathbb{R}^2\rightarrow\mathbb{R}^2$ in the following form:
\begin{equation}\label{e9}
\!\!	\begin{aligned}
		G_i(Q_i,\hat\eta_i)&\!=\!\nabla_{Q_i}C_i(Q_i,\hat\eta_i)=Q_i(t)-\hat\eta_i(t)-\delta_i^*,\\
		\Psi_i(Q_i,\hat\eta_i)&\!=\!\nabla_{\hat\eta_i(t)}C_i(Q_i,\hat\eta_i)\!=\!-(Q_i(t)\!-\!\hat\eta_i(t)\!-\! \delta_i^*).
	\end{aligned}
\end{equation}
Rewriting \eqref{e8} and \eqref{e9} in a compact form, we have
\begin{eqnarray*}
	\begin{aligned}
		&\mathcal{C}(Q,\hat\eta)\!=\! \text{col}(\mathcal{C}_1(Q_1,\hat\eta_1),\!\cdots\!,\mathcal{C}_N(Q_N,\hat\eta_N)),\\
		&G(Q,\hat\eta)\!=\!\text{col}(G_1(Q_1,\hat\eta_1),\!\cdots\!,G_N(Q_N,\hat\eta_N)),
	\end{aligned}
\end{eqnarray*}
where $\hat\eta(t)=\text{col}(\hat\eta_1(t),\cdots,\hat\eta_N(t))\in \mathbb{R}^{2N}$.

By using the predefined mappings in \eqref{e9}, the desired dynamics of the $i$-th USV are designed as follows:
\begin{equation}\label{A1}
\!\!	\begin{aligned}
		\dot Q_i(t)&\!=\!-G_i(Q_i,\hat\eta_i),\\
		\dot{\hat\eta}_i(t)&\!\!=\!\!-\gamma_i\!\!\sum_{j\in\mathcal{N}_i}\!\!(\hat\eta_i(t)\!-\!\hat\eta_j(t))\!-\!\Psi_i(Q_i,\hat\eta_i)\!-\! k_{1i}\omega_i(t),\\
			\dot {\omega}_i(t)&\!\!=\!\! k_{1i}\big(\gamma_i \!\!\sum_{j\in\mathcal{N}_i}\!\!(\hat\eta_i(t)\!-\!\hat\eta_j(t))\!+\!\Psi_i(Q_i,\hat\eta_i)\big)\!\!-\! k_{2i}\omega_i(t),\\
	\end{aligned}
\end{equation}
where $Q_i(0)$ is the initial position of the $i$-th USV, $\hat\eta_i(t)$ is the estimation of $Q_{ave}(t)$, the auxiliary variable $\omega_i(t)$ is designed to improve the convergence accuracy.
 The initial values $\hat\eta_i(0)$ and $\omega_i(0)$ can be chosen arbitrary. Moreover, the constant parameters satisfy 
 \begin{equation}\label{e11}
 	\gamma_i>0, \quad 1> k_{1i}>0,\quad k_{2i}=1-k_{1i}^2.
 \end{equation}

By defining the augmented state vector $X_i(t)=\text{col}(Q_i(t),\hat\eta_i(t),\omega_i(t))\in\mathbb{R}^6$, we shape the total energy of the $i$-th USV as $H_{di}(X_i(t)):\mathbb{R}^6\rightarrow \mathbb{R}$ in the following form:
\begin{equation*}
	\begin{aligned}
		H_{di}(X_i(t))&=\mathcal{C}_i(Q_i(t),\hat\eta_i(t))+\frac{1}{2}\omega_i(t)^\top \omega_i(t)\\&+\frac{\gamma_i}{2}\sum_{j\in\mathcal{N}_i}(\hat\eta_i(t)-\hat\eta_j(t))^\top(\hat\eta_i(t)-\hat\eta_j(t)).
	\end{aligned}
\end{equation*}
Then, leveraging the energy-based interpretation and power-preserving interconnections of PH systems, the desired dynamics \eqref{A1} can be rewritten as
\begin{equation}\label{e12}
	 \dot{X}_i(t)
	=(J_{di}-R_{di})
	\frac{\partial H_{di}(X_i(t))}{\partial X_i(t)},
\end{equation}
where 
the interconnection matrix $J_{di}\in\mathbb{R}^{6\times 6}$ and the dissipation matrix $R_{di}\in\mathbb{R}^{6\times 6}$ are respectively defined as
$$J_{di}\!=\!\!
  \begin{bmatrix}
	0_{2\times 2}&\!\!\!\! 0_{2\times 2} & \!\!\!\!0_{2\times 2} \\ 0_{2\times 2} &\!\!\!\! 0_{2\times 2}& \!\!\!\!\!-k_{1i}\otimes I_2\\
	0_{2\times 2} & \!\! k_{1i}\otimes I_2 & \!\!\!\! 0_{2\times 2} \\
\end{bmatrix}\!\!,
R_{di}\!=\!\!
\begin{bmatrix}
I_2&\!\!\!\! 0_{2\times 2} &\!\!\!\! 0_{2\times 2} \\ 
0_{2\times 2} &\!\!\!\! I_2&\!\!\!\! 0_{2\times 2}\\
0_{2\times 2} &\!\!\!\! 0_{2\times 2} & \!\!\!\! k_{2i}\otimes I_2 \\
\end{bmatrix}\!\!,
$$
with properties $J_{di}=-J_{di}^\top$, $R_{di}=R_{di}^\top\succ 0$.

Define the stacked vector $X(t)=\text{col}(Q(t),\hat\eta(t),\omega(t))$, $\Psi(Q,\hat\eta)=\text{col}(\Psi_1(Q_1,\hat\eta_1),\cdots,\Psi_N(Q_N,\hat\eta_N))$ and  $\omega(t)=\text{col}(\omega_1(t),\cdots,\omega_N(t))$, then the desired dynamics \eqref{e12} for $N$-USVs are rewritten in compact form as
\begin{equation}\label{cls}
	\begin{aligned}
	\dot X(t)&\!=\!\begin{bmatrix}
		G(Q,\hat\eta)\\
		-(\gamma L\otimes I_2)\hat\eta(t)-\Psi(Q,\hat\eta)-k_1w(t)\\ k_1((\gamma L\otimes I_2)\hat\eta(t)+\Psi(Q,\hat\eta))-k_2w(t)
	\end{bmatrix}\\
	&\!=\!(J_d-R_d) \frac{\partial H_{d}(X(t))}{\partial X(t)}\\
	\end{aligned}
\end{equation}
where the constant matrices $k_1=\text{diag}(k_{11},\cdots,k_{1N})\otimes I_2$,  $k_2=\text{diag}(k_{21},\cdots,k_{2N})\otimes I_2$, $\gamma=\text{diag}(\gamma_1,\cdots,\gamma_N)$, the skew-symmetric interconnection matrix $J_d=\text{diag}(J_{d1},$ $\cdots,J_{dN})$, the positive semi-definite dissipation matrix $R_d=\text{diag}(R_{d1},\cdots,R_{dN})$, and the total Hamiltonian function is
\begin{eqnarray}\label{h}
	\begin{aligned}
		H_{d}(X(t))&=\!\sum_{i=1}^N\!\big(\mathcal{C}_i(Q_i(t),\hat\eta_i(t))\!+\!\frac{1}{2}\omega_i(t)^\top \omega_i(t)\\&+\frac{\gamma_i}{2}\hat\eta_i(t)^\top\!\sum_{j\in\mathcal{N}_i}\!\!(\hat\eta_i(t)-\hat\eta_j(t))\big).
	\end{aligned}
\end{eqnarray}
The design of the Hamiltonian $H_d(X(t))$ and the corresponding dynamics \eqref{cls} are constructed based on the following three key components:
\begin{itemize}\item The term $\sum_{i=1}^N \mathcal{C}_i(Q_i,\hat\eta_i)$ represents the potential energy injected to drive the system toward the solution of the optimization problem \eqref{p1}. According to \eqref{e8}, the minimum of $\mathcal{C}_i(Q_i,\hat\eta_i)$ is achieved if and only if the steady-state position $\lim_{t\rightarrow\infty} Q_i(t)=Q_i^*$ and the estimation $\lim_{t\rightarrow\infty} \hat\eta_i(t)=\hat\eta_i^*$ of the $i$-th USV satisfy $Q_i^*-\hat\eta_i^* = \delta_i^*$. When  $\hat\eta_i^*=Q_{ave}^*$, the desired formation is reached.
\item The term $\frac{\gamma_i}{2}\hat\eta_i(t)^\top\sum_{j\in\mathcal{N}_i}(\hat\eta_i(t)-\hat\eta_j(t))$ represents the coupling potential, which penalizes the estimation disagreement and drives the local estimation $\hat\eta_i(t)$ toward a common consensus value $\hat\eta^*$. 
\item The quadratic term $\frac{1}{2}\omega_i(t)^\top \omega_i(t)$ functions as an internal integrator. It is designed to improve the precision of the designed dynamics.
\end{itemize}

Moreover, the predefined parameters $\gamma_i$, $k_{1i}$, and $k_{2i}$ are assigned heterogeneously to the $i$-th USV as local control gains. Such a distributed configuration not only modulates the convergence rate of the closed-loop system but also significantly enhances privacy preservation. Since these heterogeneous local gains are kept private to each USV, the internal control strategy remains obscured from both HBC adversaries and external eavesdroppers. A rigorous stability analysis of the system under these gains is presented in the next section.

\subsection{Convergence analysis of the designed dynamics}
In this subsection, we analyze the convergence of the designed dynamics \eqref{cls}. It should be noted that \eqref{cls} serves as a distributed generator for the reference trajectories, rather than a controller. The analysis focuses on ensuring that the generated trajectories satisfy the formation objective \eqref{e5}.
\begin{theorem}\label{t1}
	For any given initial condition $X(0)=\text{col}(Q(0),\hat\eta(0),\omega(0))$, the trajectory generated by the designed dynamics \eqref{cls} converges at an exponential rate to a steady state $X^*$. Furthermore, this steady state $X^*$ is an optimal solution of the formation problem \eqref{p1}.
\end{theorem}
\begin{proof}
Choosing the Hamiltonian function $H_{d}(X(t))$ in (\ref{h}) as a candidate Lyapunov function:
\begin{equation}\label{lya}
	\begin{aligned}
		&L_y(Q(t),\hat\eta(t),\omega(t))=1_N^\top \mathcal{C}(Q(t),\hat\eta(t))\\&+\frac{1}{2}\hat\eta(t)^\top(\gamma L\otimes I_2)\hat\eta(t)+\frac{1}{2}\omega(t)^\top \omega(t)\geq 0.
	\end{aligned}
\end{equation}
To further analyze the convergence, we introduce the following orthogonal transformation of the Laplacian matrix:
$$\eta(t) =
\begin{bmatrix}
	\eta_1(t) \\ \eta_2(t)
\end{bmatrix}=
\begin{bmatrix}
	r^\top \otimes I_2 \\ R^\top \otimes I_2
\end{bmatrix} \hat\eta(t),$$
where $r=\frac{1}{\sqrt{N}}1_N$, $r^\top R=0_{N-1}^\top$, $R^\top R=I_{N-1}$, $RR^\top=I_N-\frac{1}{N}1_N1_N^\top$, $\eta_1\in \mathbb{R}^2$ and $\eta_2\in \mathbb{R}^{2(N-1)}$. Then, we have 
\begin{eqnarray*}
	\begin{aligned}
		\hat\eta(t)^\top (\gamma L\otimes I_2)\hat\eta(t) &= \hat\eta(t)^\top (RR^\top \gamma L RR^\top\otimes I_2) \hat\eta(t)\\
		&=\eta_2(t)^\top (R^\top \gamma L R\otimes I_2)\eta_2(t).
	\end{aligned}
\end{eqnarray*}
Hence, $L_y(Q(t),\hat\eta(t),\omega(t))$ converts to
\begin{eqnarray}\label{lya1}
	\begin{aligned}
		&L_y(Q(t),\hat\eta(t),\omega(t))=1_N^\top \mathcal{C}(Q(t),\hat\eta(t))
		\\&+\frac{1}{2}\omega(t)^\top\omega(t)+\frac{1}{2}\eta_2(t)^\top (R^\top\! \gamma L R\otimes\! I_2)\eta_2(t)\geq 0.
	\end{aligned}
\end{eqnarray}
Next, we show that $\dot L_y(Q(t),\hat\eta(t),\omega(t))\leq 0.$ The dynamic system \eqref{cls} can be rewritten as follows
\begin{equation}\label{te}
\!\!\begin{cases}
		\dot Q(t) = -G(Q,\hat\eta),\\
		\dot {\hat\eta}(t)=-\Psi(Q,\hat\eta)-(\gamma LR\otimes I_2) \eta_2(t)-k_1 \omega(t),\\
		\dot{ \eta}_2 (t)\!=\! -(R^\top\otimes I_2)\Psi(Q,\hat\eta)\!-\!(R^\top \gamma L R\otimes I_2) \eta_2(t)\\
		\quad\quad-k_1 (R^\top\otimes I_2)\omega(t),\\
		\dot \omega(t) = k_1\big(\Psi(Q,\hat\eta)+(\gamma LR\otimes I_2)\eta_2(t)\big)-k_2 \omega(t).
	\end{cases}
\end{equation}
Taking the time derivative of $L_y(Q(t),\hat\eta(t),\omega(t))$ in \eqref{lya1} yields (\ref{te}), we obtain 
\begin{equation}\label{dl}
	\begin{aligned}
		&\dot L_y(Q(t),\hat\eta(t),\omega(t))=-G(Q(t),\hat\eta(t))^\top G(Q(t),\hat\eta(t))\\&-\Psi(Q(t),\hat\eta(t))^\top\Psi(Q(t),\hat\eta(t))\\
		&-((R^\top \gamma L R\otimes I_2)\eta_2(t))^\top (R^\top \gamma L R\otimes I_2)\eta_2(t)\\
		&
		-2\Psi(Q(t),\hat\eta(t))^\top (\gamma LR\otimes I_2)\eta_2(t)-k_2\omega(t)^\top \omega(t).
	\end{aligned}
\end{equation}
With the definition in (\ref{e9}), we have
\begin{equation*}
	\begin{aligned}
	\Vert G(Q,\hat\eta)\Vert^2=\Vert\Psi(Q,\hat\eta)\Vert^2
	=\sum_{i=1}^N\Vert Q_i(t)-\hat\eta_i(t)-\delta_i^*\Vert^2.
	\end{aligned}
\end{equation*}
Hence, (\ref{dl}) implies
\begin{equation*}\label{dl1}
	\begin{aligned}
		&\dot{L}_y(Q(t),\hat\eta(t),\omega(t))\leq-2\Psi(Q,\hat\eta)^\top\Psi(Q,\hat\eta)\\
		&-((R^\top \gamma L R\otimes I_2)\eta_2(t))^\top (R^\top \gamma L R\otimes I_2)\eta_2(t)\\
		&-2\Psi(Q,\hat\eta)^\top (\gamma LR\otimes I_2)\eta_2(t)-k_2\omega(t)^\top \omega(t).
	\end{aligned}
\end{equation*}
By applying the Young's inequality, one get
\begin{eqnarray}\label{dl2}
	\begin{aligned}
	&	\dot{L}_y(Q(t),\hat\eta(t),\omega(t))\\&\leq-\frac{1}{2} \Vert\Psi(Q,\hat\eta)\Vert^2-\frac{\gamma^2_o}{3} \Vert \eta_2(t)\Vert^2-k_2^o\Vert \omega(t)\Vert^2\\
		&=-1_N^\top \mathcal{C}(Q,\hat\eta)-\frac{2k_2^o}{2}\omega(t)^\top \omega(t)-\frac{2\gamma_o}{3}\frac{\gamma_o}{2}\Vert\eta_2(t)\Vert^2
		\\&\leq -\min\{1,2k_2^o,\frac{2\gamma_o}{3}\}L_y(Q(t),\hat\eta(t),\omega(t)),
	\end{aligned}
\end{eqnarray}
where $\gamma_o=\min\{\gamma_1,\cdots,\gamma_N\}$, $k_2^o=\min\{k_{21},\cdots,k_{2N}\}$. Hence, the designed dynamics \eqref{cls} exponentially converges to $X^*$, with a decay rate of at least $\min\{1,2k_2^o,\frac{2\gamma_o}{3}\}$.


Recalling \eqref{A1}, the steady-state $X^*=\text{col}(Q^*, \hat\eta^*, \omega^*)$ satisfies the following set of equations
\begin{subequations}\label{e20}
	\begin{align}
		G(Q^*,\hat\eta^*)=0_{2N},\\
		(\gamma L\otimes I_2)\hat\eta^*+\Psi(Q^*,\hat\eta^*)=0_{2N},\\
		\omega^*=0_{2N}.
	\end{align}
\end{subequations}
For an undirected and connected graph $\mathcal{G}$, the associated Laplacian matrix $L$ possesses a unique zero eigenvalue with a corresponding eigenvector $1_N$, which implies $1_N^\top L = 0$. By substituting \eqref{e9} into (\ref{e20}a), we obtain $$G_i(Q_i^*,\hat\eta_i^*) = -\Psi_i(Q_i^*, \hat\eta_i^*)=0_{2}.$$ 
Consequently, according to (\ref{e20}b), it holds that $(\gamma L \otimes I_2) \hat\eta^* = 0$. Given $\gamma \succ 0$, this property directly implies that all local estimations reach a consensus, i.e., $$\hat\eta_i^* = \hat\eta_j^*,\quad \forall i, j \in \mathcal{V}.$$ 
Furthermore, since $G_i(Q_i^*, \hat\eta_i^*) = Q_i^* - \hat\eta_i^* - \delta_i^* = 0_2$ follows from (\ref{e20}a), summing over all USVs yields $\sum_{i=1}^N Q_i^* - \sum_{i=1}^N \hat\eta_i^* = \sum_{i=1}^N \delta_i^*$. Invoking the definition of formation offsets $\delta_i^*$, where $\sum_{i=1}^N \delta_i^* = 0$ is naturally satisfied by $\delta_i^*=\frac{1}{N}\sum_{i=1}^N\delta_{ij}^*$, we conclude that $$\hat\eta_i^* = \frac{1}{N} \sum_{i=1}^N Q_i^* = Q_{ave}^*,\quad \forall i \in \mathcal{V}.$$
In light of the definition of $C_i(Q_i, \hat\eta_i)$ in (\ref{e8}), the attainment of $\hat\eta_i^* = Q_{ave}^*$ ensures that 
$$V_i(Q_i^*, Q_{ave}^*) = C_i(Q_i^*, \hat\eta_i^*) = 0_2, \quad \forall  i \in \mathcal{V}.$$ 
 The gradient of $V_i(Q_i, Q_{ave})$ with respect to $Q_i(t)$, evaluated at $Q_i^*$, is given by:
 \begin{equation*}
 	\nabla_{Q_i} V_i(Q_i, Q_{ave}) \vert_{Q_i=Q_i^*} = 0_2,  \quad \forall  i \in \mathcal{V}.
 	\end{equation*}
 Based on Lemma \ref{l1}, it is established that $Q^*$ is a minimum point of the cost function $V_i(Q_i, Q_{ave})$ in (\ref{e6}). 
\end{proof}

\subsection{Controller implementation}
Based on the stability and convergence established in Section \ref{S2}, the dynamics \eqref{cls} serve as a distributed reference generator for the formation task. To realize the physical deployment, a hierarchical control architecture is adopted where $Q_i^*(t)$ provides the time-varying reference trajectory. The objective of this section is to design a local tracking law to ensure that the actual position trajectory $Q_i(t)$ asymptotically follow $Q_i^*(t)$ under unknown bounded disturbances $\tau_{di}$, thereby achieving the desired formation in real-world environment. The overall control scheme of the $i$-th USV is given in Fig.~\ref{ff}.
 \begin{figure}[!t]
	\centering
	\includegraphics[width=3in]{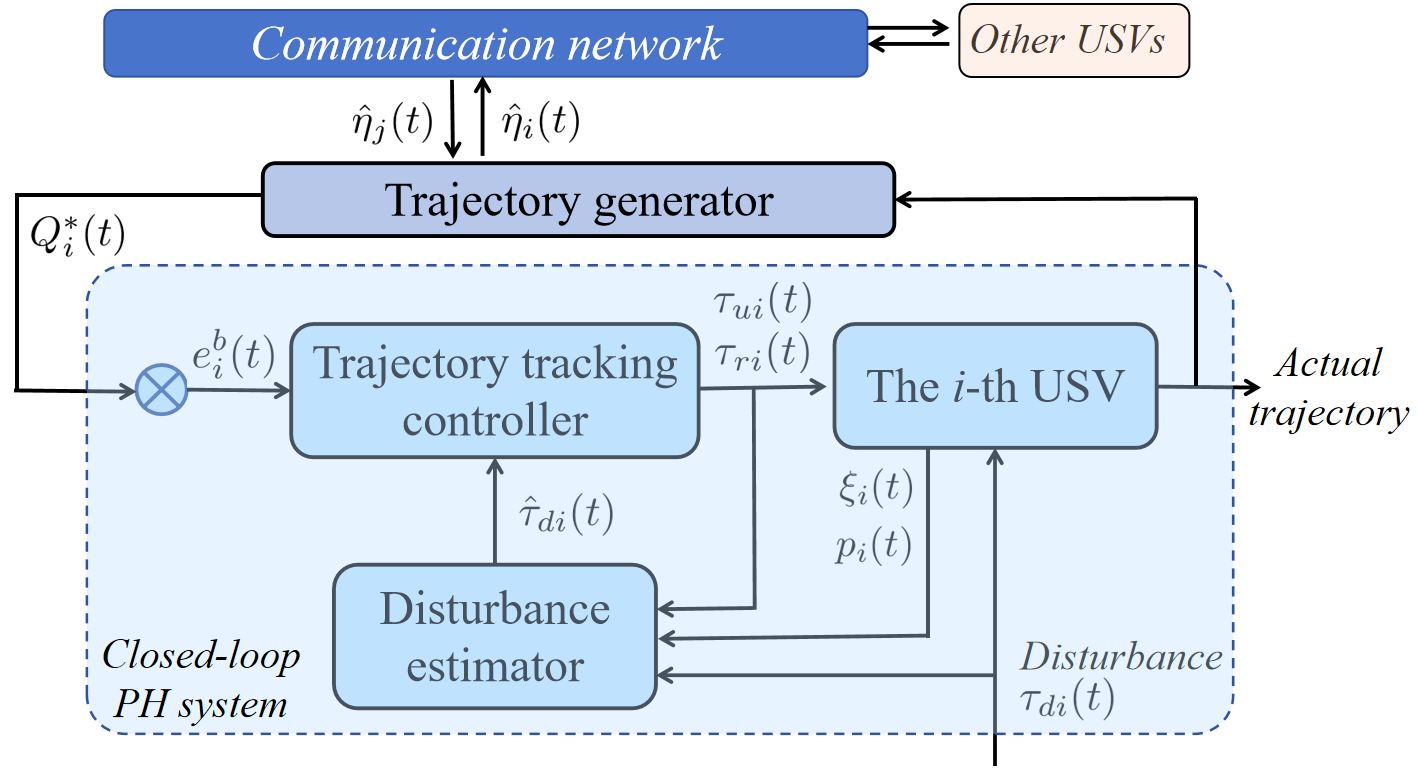}
	\caption{The control scheme of the $i$-th USV.}
	\label{ff} 
\end{figure}
\subsubsection{Coordinate transformation}
To facilitate the PH formulation while maintaining the geometric structure of the tracking task, the $i$-th USV's actual and desired positions are projected onto the instantaneous body-fixed frame. We define the following coordinate transformations:
\begin{equation}\label{e1111}
	\begin{bmatrix} \xi_i^1(t) \\ \xi_i^2(t) \end{bmatrix} \!\!=\! R^\top(\psi_i) \begin{bmatrix} x_i(t) \\ y_i(t) \end{bmatrix}, 
	\begin{bmatrix} \xi_{di}^1(t) \\ \xi_{di}^2(t) \end{bmatrix} \!\!=\! R^\top(\psi_i) \begin{bmatrix} x_i^*(t) \\ y_i^*(t) \end{bmatrix},
\end{equation}
where $\xi_i^1(t)$ and $\xi_i^2(t)$ represent projections of the inertial position onto the longitudinal and lateral axes of the current body-fixed frame, while $\xi_{di}^1(t)$ and $\xi_{di}^2(t)$ denote the projection of the time-varying reference trajectories $x_i^*(t)$, $y_i^*(t)$ onto the same frame. It is important to note that, unlike the inertial coordinates $(x_i, y_i)$, the variable $(\xi_i^1, \xi_i^2)$ is coupled with the heading angle $\psi_i$. Let $\xi_i(t)=[\xi_i^1(t),\xi_i^2(t),\psi_i(t)]^\top\in\mathbb{R}^3$, then its time derivatives inherently incorporate the rotational kinematics:
$$\dot\xi_i(t)=\begin{bmatrix} \dot{\xi}_i^1(t) \\ \dot{\xi}_i^2(t)\\ \dot \psi_i(t) \end{bmatrix} =\underbrace{\begin{bmatrix} 1 & 0 & \xi_i^2(t)\\ 0 & 1& -\xi_i^1(t) \\ 0& 0& 1\end{bmatrix}}_{J_i^\xi(\xi_i)} \begin{bmatrix} u_i(t) \\ v_i(t)\\ r_i(t) \end{bmatrix}.$$
With this formulation, the PH system \eqref{e3} becomes \cite{rf33}
\begin{equation}\label{w3}
	\begin{aligned}
		\begin{bmatrix}
			\dot{\xi}_i(t)	\\ \dot p_i(t)
		\end{bmatrix}
		&=\begin{bmatrix}
			0 & J_i^\xi(\xi_i)\\ -(J_i^\xi(\xi_i))^\top &-\bar D_i(p_i)
		\end{bmatrix}\begin{bmatrix}
			\nabla_{\xi_i}H^b_i\\ \nabla_{p_i}H^b_i
		\end{bmatrix}
		\\&+\begin{bmatrix}
			0_{3\times 3}\\ I_3
		\end{bmatrix}(\tau_c(t)+\tau_d(t)),
	\end{aligned}
\end{equation}
where $H_i^b(\xi_i, q_i) = \frac{1}{2} p_i^\top M_i^{-1} p_i + V_i^b(\xi_i)$ denotes the energy function with the potential energy $V_i^b(\xi_i)=V_i(q_i)$.


\subsubsection{Adaptive estimator}
Based on the immersion and invariance method in \cite{r1}, an adaptive estimator for the bounded unknown disturbance $\tau_{di}(t)$ is designed as follows:
\begin{equation}\label{w4}
	\begin{aligned}
		\hat{\tau}_{di}(t)&=\alpha_i\int_0^t -(J_i^\xi(\xi_i))^\top\nabla_{\xi_i} H_i^b -\bar D_i(p_i)\nabla_{p_i} H_i^b\\&+\tau_{ci}(\tau)+\hat{\tau}_{di}(\tau)d\tau+\alpha_ip_i(t),
	\end{aligned}
\end{equation}
where the tuning parameter $\alpha_i>0$.
\subsubsection{Controller design}
As illustrated in Fig.~\ref{f1}, the propulsion system consists of two symmetric stern propellers which only provide surge force and yaw moment. 
Due to the absence of a tunnel thruster, the USV is underactuated with the control input $\tau_{ci} = [\tau_{ui}, 0, \tau_{ri}]^\top$. To address the resulting non-collocated stabilization challenge, a virtual reference point $P_i$ is defined at a look-ahead distance $\beta_i>0$ from the center of mass. As noted in \cite{rf33}, this coordinate shifting introduces a geometric coupling between the yaw rate and the lateral motion of $P_i$, allowing the yaw input to indirectly regulate the lateral tracking error despite the lack of direct sway actuation. This transformation effectively recasts the underactuated task into a fully actuated tracking problem at the virtual point within the PH framework.

By shifting the tracking objective to the virtual point $P_i$, we redefine the tracking error $e_i^b(t)$ for the $i$-th USV as:
\begin{equation}\label{w1}
	e_i^b(t)=\begin{bmatrix} e^s_i(t) \\ e^w_i(t) \end{bmatrix}=	\begin{bmatrix} \xi_i^1(t)+\beta_i \\ \xi_i^2(t) \end{bmatrix}-\begin{bmatrix} \xi_{di}^1(t) \\ \xi_{di}^2(t) \end{bmatrix},
\end{equation}
where $e^s_i(t), e^w_i(t)\in\mathbb{R}$ denote the surge and sway errors of the look-ahead point, respectively, the reference trajectories $\xi_{di}^1(t)$ and $\xi_{di}^2(t)$ is obtained by \eqref{A1} and \eqref{e1111} . The objective of this section is to design a IDA-PBC\footnote{Interconnection and Damping Assignment Passivity-Based Control (IDA-PBC) \cite{rf21} achieves stabilization by designing a control law that matches the open-loop dynamics with a desired PH structure characterized by a specific energy function and dissipation.} controller such that
\begin{equation}\label{w2}
	\lim_{t\rightarrow\infty} e_i^b(t)=0.
\end{equation}

First, we design the desired error system as follows:
\begin{equation}\label{w7}
	\begin{aligned}
		\begin{bmatrix}
			\dot e^b_i(t)\\ \dot{\tilde p}_i(t)
		\end{bmatrix}&=\begin{bmatrix}
			S_i^{11}&S_i^{12}\\(S_i^{12})^\top&S_i^{22}
		\end{bmatrix}\begin{bmatrix}
			\nabla_{e^b_i}H_i^e\\ \nabla_{\tilde p_i}H_i^e
		\end{bmatrix},
	\end{aligned}
\end{equation}
where $S_i^{11}\in \mathbb{R}^{2\times 2}$ and $S_i^{22}\in \mathbb{R}^{3\times 3}$ are negative semi-definite matrices, $\tilde p_i(t)=p_i(t)-p_{di}(t)$ denotes the velocity error with the reference velocity $p_{di}(t)=\text{col}(u_{di}(t),v_{di}(t),r_{di}(t))$, the desired error Hamiltonian function is designed as
\begin{equation}\label{h1}
	H_i^e=\frac{1}{2}\tilde p_i(t)^\top M_i^{-1}\tilde p_i(t)+(e_i^b(t))^\top K_{di}e_i^b(t),
\end{equation}
where $K_{di}=\text{diag}(K_{di}^1,K_{di}^2)\in\mathbb{R}^{2\times 2}$ is positive definite . 
Differentiating both sides of \eqref{w1}, substituting \eqref{e3}, \eqref{w7} yields
\begin{equation}\label{w0}
	\begin{aligned}
	\dot{e}_i^b(t)&=d(R^\top(\psi_i)(Q_i(t)-Q_i^*(t))+\Delta_i)/dt\\
	&=S_i^{11}K_{di}e_i^b(t)+S_i^{12}M_i^{-1}\tilde p_i(t),
	\end{aligned}
\end{equation}
where the vector $\Delta_i=[\beta_i,0]^\top\in\mathbb{R}^2$. We set $S_i^{11}=\begin{bmatrix}
	-1&\frac{r_i}{2K_{di}^2}\\-\frac{r_i}{2K_{di}^1} &-1
\end{bmatrix}$ and $S_i^{12}=\begin{bmatrix}
1&0&0\\0&0&\beta_i
\end{bmatrix}$ to eliminate the term $\mathcal{T}(r_i)(R^\top(\psi_i)Q_i(t)+\Delta_i)$ and $M_i^{-1}\tilde p_i(t)$, respectively, where $\mathcal{T}(r_i)=\begin{bmatrix}
0&-r_i\\r_i&0
\end{bmatrix}$. 
Then, by defining  $\Omega_i=\text{diag}(1,\beta_i)$ and solving \eqref{w0}, the reference velocity is obtained as
$$\begin{bmatrix}{u}_{di}(t)\\ r_{di}(t)\end{bmatrix}
=\Omega_i^{-1}(\mathcal{T}(r_i)+2S^{11}_iK_{di})e_i^b(t)-\begin{bmatrix}0\\v_i
\end{bmatrix}
+R^\top(\psi_i)\dot{Q}_i^*(t).$$

The IDA-PBC controller is further designed as follows:
\begin{equation}\label{u1}
	\begin{aligned}
&	\begin{bmatrix}
		\tau_{ui}(t)\\ \tau_{ri}(t)
	\end{bmatrix}\!=\! W\big(\bar D_i(p_i)M_i^{-1}p_i(t)\!-\!2(S^{12}_i)^\top K_{di}e_i^b+\\& \!2K_{di}(S^{12}_i)^\top\Delta_i\!+\! S^{22}_iM_i^{-1}\tilde{p}_i-\hat\tau_{di}+M_iW^{-1}\begin{bmatrix}\dot{u}_{di}\\ \dot{r}_{di}\end{bmatrix}
	\big),
	\end{aligned}
\end{equation}
where the negative definite parameter matrix $S_i^{22}\in\mathbb{R}^{3\times 3}$ denotes the damping injection.

\begin{theorem}
	Consider the USV system \eqref{w3} subject to bounded external disturbances $\vert \tau_{di}(t)\vert\leq \tau_{max}$. Suppose the control law is given by \eqref{u1} with the estimator \eqref{w4}, then the closed-loop tracking error is uniformly ultimately bounded.  
\end{theorem}
\begin{proof}
	Define the disturbance estimator error as $e^d_i(t)=\hat\tau_{di}(t)-\tau_{di}(t)$, then taking the time derivative of it and substituting \eqref{w3} and \eqref{w4} yields
	\begin{equation}\label{w5}
		\dot{e}^d_i(t)=-\alpha_i e^d_i(t)-\dot {\tau}_{di}(t).
	\end{equation}
	Choosing the candidate Lyapunov function as 
	$$L_i(e^b_i,\tilde p_i,e^d_i)=H_i^e+\frac{1}{2}(e^d_i(t))^\top e^d_i(t)\geq 0.$$
	The time derivative of $L_i(e^b_i,\tilde p_i,e^d_i)$ along \eqref{w7} and \eqref{w5} is 
	\begin{equation*}
		\begin{aligned}
		\dot{L}_i&=(M_i^{-1}\tilde p_i(t))^\top S^{22}_iM_i^{-1}\tilde p_i(t)+4(K_{di}e_i^b(t))^\top S^{11}_iK_{di}e_i^b(t)\\
		&-\alpha_i(e^d_i(t))^\top e^d_i(t)-(e^d_i(t))^\top \dot{\tau}_{di}
			\leq -c_{1i} L_i +c_{2i},
		\end{aligned}
	\end{equation*}
	where $c_{1i}=\min\{-2\frac{\lambda_{\max}(M_i^{-\top}S_i^{22}M_i^{-1})}{\lambda_{\max}(M_i^{-1})},-\frac{\lambda_{\max}(K_{di}^{\top}S_i^{11}K_{di})}{\lambda_{\max}(K_{di})}$, $2\alpha_i-2\varepsilon_i\}>0$, $c_{2i}=\frac{\tau^2_{\max}}{4\varepsilon_i}$ with $\alpha_i>\varepsilon_i>0$.
	
	Hence, we have 
	$$0\leq L_i\leq \frac{c_{2i}}{c_{1i}}+(L_i(0)-\frac{c_i^2}{c_i^1})e^{-c_i^1t},$$
	which implies that the system error converge to a neighborhood of the origin with the radius $\sqrt{\frac{c_{2i}}{c_{1i}}}$. 
\end{proof}
\subsection{Privacy analysis}
Given the initial states as $Q(0)=\text{col}(Q_1(0),\cdots,$ $Q_N(0))$, $\hat\eta(0)=\text{col}(\hat\eta_1(0),\cdots,\hat\eta_N(0))$, $\omega(0)=\text{col}(\omega_1(0),$ $\cdots,\omega_N(0))$, 
if no collusion occurs, the information set accessible to the $h$-th HBC adversary is defined by $$\mathcal{I}_h=\{Q_h,\hat\eta_h,\omega_h,\delta_h^*,\hat\eta_l\mid l\in\mathcal{N}_h\}.$$

According to \eqref{Eh1}, HBC adversaries can obtain more information through collusion, which facilitates the inference of neutral USVs' private information. Therefore, HBC adversaries will inevitably collude when they are neighbors. 
Through collusion between HBC adversaries $h$ and $k$, the combined information set accessible to them is as follows:
\begin{equation}\label{Eh2}
\mathcal{I}^c_h(t) =  \mathcal{I}_h(t) \cup\big( \cup_{k \in \mathcal{N}_h \cap \mathcal{H}} \mathcal{I}_k(t) \big)..
\end{equation}
In this case, the information set accessible to the external adversary is denoted by
\begin{equation}\label{I1}
	E^c(t) = \{\mathcal{A} \} \cup \{\hat\eta_i(t)\mid i \in \mathcal{V}\} \cup \{{ \mathcal{I}^c_h(t) \mid h \in \mathcal{H} }\}.
\end{equation}
If the initial states of the $r$-th neutral USV $(r\in\mathcal{V} \setminus \mathcal{H})$ varying from $Q_r(0),\hat\eta_r(0),\omega_r(0)$ to $Q_r'(0),\hat\eta_r'(0),\omega_r'(0)$ while initial states of other USVs are fixed, we have 
\begin{equation*}
\begin{aligned}
	Q'(0)&\!=\!\text{col}(Q_1(0),\cdots\!,Q_{r-1}(0),Q_r'(0),Q_{r+1}(0)\cdots\!,Q_N(0)),\\
	\hat\eta'(0)&\!=\!\text{col}(\hat\eta_1(0),\cdots,\hat\eta_{r-1}(0),\hat\eta_r'(0),\hat\eta_{r+1}(0)\cdots,\hat\eta_N(0)),\\
	\omega'(0)&\!=\!\text{col}(\omega_1(0),\cdots,\omega_{r-1}(0),\omega_r'(0),\omega_{r+1}(0)\cdots,\omega_N(0)),
\end{aligned}
\end{equation*}
the corresponding information set accessible to the $h$-th HBC adversary (or external eavesdroppers) becomes $\mathcal{I}_h^{c'}(t)$ ($E^{c'}(t)$). 
Motivated by \cite{rf29}, the following definition is presented:
\begin{definition}\label{pd}
	The privacy $Q_r(t)$ of the $r$-th neutral USV $(r\in\mathcal{V} \setminus \mathcal{H})$ is preserved if, for any trajectory $Q_r(t)$, there exists an alternative $Q_r'(t)$ such that the information sets accessible to the HBC adversaries (or external eavesdroppers) are identical, i.e., $\mathcal{I}_h^c(t) = \mathcal{I}_h^{c'}(t)$ (or $E^c(t)=E^{c'}(t)$), $\forall t \geq 0$.
\end{definition}
\begin{theorem}\label{ppp1}
	The privacy of the $r$-th neutral USV $(r\in\mathcal{V} \setminus \mathcal{H})$ is preserved against both HBC adversaries and external eavesdroppers, even neighboring HBC adversaries are colluded.
\end{theorem}
\begin{proof}
	Substituting \eqref{e9} into \eqref{A1} and defining the coordinate transformation $\theta(t) = Q(t)- k_1\omega(t)$,  then the multi-USV system's dynamics \eqref{cls} can be rewritten as 
	\begin{equation}\label{ety1}
		\begin{cases}
			\dot{\theta}(t)\!=\!-k_2\theta(t)\!+\!(2I_2\!+\!k_1^\top k_1\gamma(L\otimes I_2))\hat{\eta}(t)\!+\!2\delta^*,\\ \dot{\hat{\eta}}(t) \!=\! -\gamma (L\otimes I_2) \hat{\eta}(t)\! + \!\theta(t)\!-\!\hat{\eta}(t)\!-\!\delta^*,
		\end{cases}
	\end{equation}
	where $\delta^*=\text{col}(\delta_1^*,\cdots,\delta_N^*)$. 
	From \eqref{ety1}, the evolution of $\hat{\eta}(t)$ is uniquely determined by the initial conditions $\theta(0)$, $\hat{\eta}(0)$ and the constant reference $\delta^*$.
	For any variation from $Q_r(0)$ to $Q_r'(0)$, 
	if USV $r$ initializes its auxiliary variable as $\omega_r'(0) = \omega_r(0) + k_1^{-1} (Q_r'(0)- Q_r(0))$, and chooses $\hat\eta_r'(0)=\hat\eta_r(0)$, then we have the initial value  $\theta_r'(0) = Q_r'(0) - k_1 \omega_r'(0)=\theta_r(0)$ remains invariant. 
	Consequently, $\theta(0)$ and $\hat{\eta}(0)$ are unchanged, leading to identical trajectories of $\hat{\eta}(t)$ for all $t \geq 0$. 
	
Note that the adversaries' accessible information set \eqref{Eh2} and \eqref{I1} depend solely on the trajectories of $\hat{\eta}(t)$ and the local variables of the HBC adversaries (which are independent of $Q_r(t)$), we conclude that
$\mathcal{I}_h^c(t) = \mathcal{I}_h^{c'}(t)$ (or $E^c(t)=E^{c'}(t)$) holds for all $t \geq 0$.
	Hence, the variation in $Q_r(t)$ is indistinguishable against HBC adversaries and external eavesdroppers, which completes the proof. 
	\end{proof}
\section{Experimental results}
This section provides a description of the experimental platform and the subsequent field implementation of the proposed hierarchical control method for USV formation coordination.  
\subsection{Experimental Setup}
\subsubsection{Hardware Platform and Sensing System}
The experimental validation is conducted using a custom-developed twin-hull USV. The USV measures $1.0\,\text{m} \times 0.6\,\text{m} \times 0.4\,\text{m}$ with a total mass of $60\,\text{kg}$ and a shallow draft of $0.1\,\text{m}$. 
 As illustrated in Fig.~\ref{f5}, each USV is equipped with
 \begin{itemize}
 	\item Localization. A dual-antenna real-time kinematic Global Navigation Satellite System (GNSS) provides centimeter-level positioning and absolute heading.
 	\item Inertial Sensing. An Inertial Measurement Unit (IMU) measures linear acceleration and angular velocity, with an internal magnetometer for yaw compensation.
 	\item Propulsion. Twin $9.5\,\text{kg}$-grade brushless thrusters provide differential thrust, controlled by 80A bi-directional Electronic Speed Controllers (ESCs) to achieve 3-DOF maneuverability.
 \end{itemize}
 \begin{figure}[!t]
 	\centering
 	\includegraphics[width=3.5in]{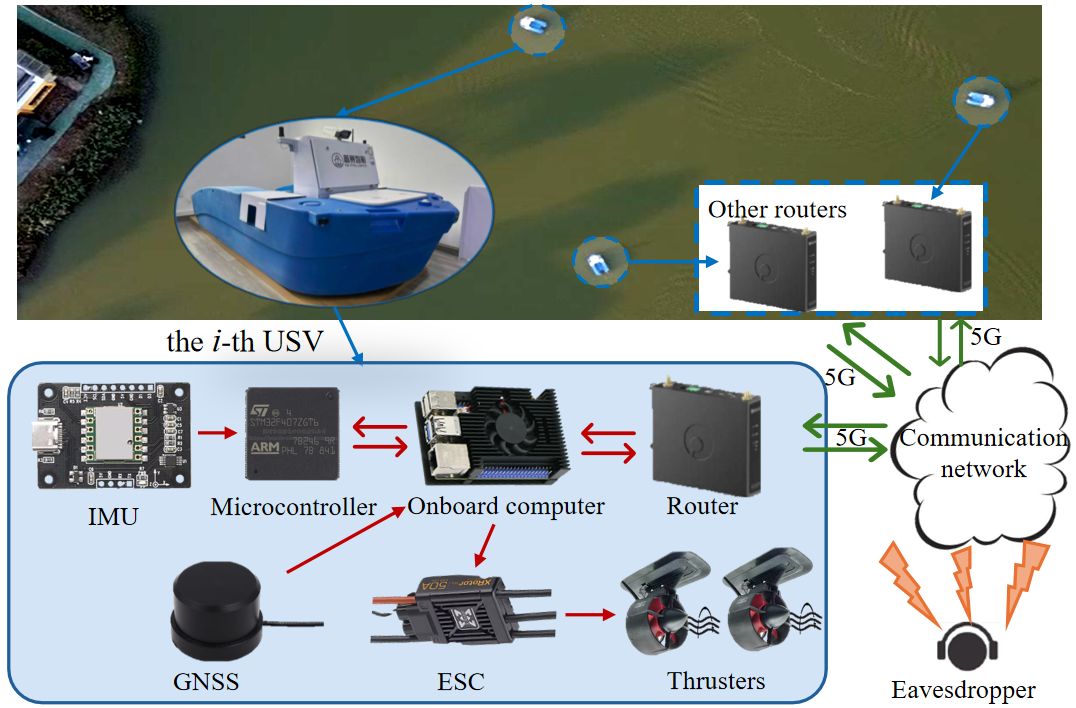}
 	\caption{Real-world experiments framework.}
 	\label{f5} 
 \end{figure}
\subsubsection{Distributed Control and Communication Architecture}
The system implements a hierarchical, distributed control architecture. The operational flow for each USV is as follows:
\begin{itemize}
	\item Communication. The USV captures its motion state via the onboard GNSS and IMU. Unlike centralized frameworks, the $i$-th USV  communicates with its neighbors via a $4G$-based Peer-to-Peer (P2P) network to exchange the estimation $\hat\eta_i(t)$.
	\item Trajectory planning. The onboard industrial computer functions as the central station for each USV. It generates planning trajectories locally based on the received neighbor information $\hat\eta_j(t)$, $j\in\mathcal{N}(i)$.
	\item Instruction Execution. Once the computing platform generates the control action, it is transmitted to the STM32F407 microcontroller via a serial bus. The microcontroller then decodes these instructions into Pulse Width Modulation signals to drive the propulsion system.
\end{itemize}
\subsubsection{Environmental Disturbance Modeling}
To rigorously evaluate the robustness of the proposed formation control strategy, the outdoor pond experiment accounts for inherent environmental uncertainties. 
The lumped disturbance $\tau_{di}$ acting on the $i$-th USV is modeled as a combination of time-varying environmental forces (wind and waves):
$$\tau_{di}(t) = \sum_{k=1}^{n} b^1_{ik} \sin(b^2_{ik} t + b^3_{ik}) + b^4_i,$$where $b^1_{ik}$, $b^2_{ik}$, and $b^3_{ik}$ represent the unknown amplitude, frequency, and phase of the $k$-th harmonic component, respectively, and $b^4_i$ denotes the bounded residual noise. In our implementation, the control system leverages the integrated IMU and GNSS data to estimate these perturbation forces as $\hat\tau_{di}(t)$, ensuring that the formation maneuvers remain stable despite the periodic oscillations of the aquatic environment.
\subsection{Experiment Results and Analysis}
To validate the effectiveness of the proposed hierarchical control method, we carry out the experiment on $4$-USVs in an outdoor pond environment (Hangzhou, China). 

The initial velocity of the $i$-th USV is $\nu_i(0)=0_3$, the initial positions of $4$-USVs and the predefined displacement vectors are given in Table~\ref{tab1}.
\begin{table}
	\begin{center}
		\caption{The predefined vectors of USVs.}
		\label{tab1}
		\begin{tabular}{| c | c | c |}
			\hline
			$\quad$ & 	Initial positions $Q_i(0)$& Desired displacement $\delta_i^*$\\		
			\hline
			$i=1$ & $[-17.12,17.76]^\top$ &$[-25.87,-25.87]^\top$\\
				$i=2$ & $[-11.91,-25.27]^\top$ & $[-17.71,-17.71]^\top$\\
				$i=3$& $[13.24,-34.10]^\top$  & $[26.33,26.33]^\top$\\
				$i=4$&$[8.68,9.45]^\top$ &  $[17.26,17.26]^\top$ \\ 
			\hline
		\end{tabular}
	\end{center}
\end{table} 
Utilizing the trajectory planning method proposed in Section~\ref{S2}, the desired quadrilateral formation trajectory is successfully generated by \eqref{cls}, as shown in Fig.~\ref{f6}.

Assume the environmental disturbances are produced by $\tau_{di}(t)=(0.01\sin(0.01t+0.01)+0.01)\otimes 1_3$. The estimator gains in \eqref{w4} are chosen as $\alpha_i=500$. The controller gains in \eqref{u1} are chosen as $S_i^{22}=-\text{diag}(1500,0,1700)$, $K_{di}=\text{diag}(0.62,0.68)$. By using the trajectory tracking controller \eqref{u1}, the actual trajectories of USVs are shown in Fig.~\ref{f8}, and the comparison between them and the planned trajectories in \eqref{cls} on $x$-axis are given in Fig.~\ref{f7}.
 \begin{figure}[!t]
	\centering
	\includegraphics[width=2.7in]{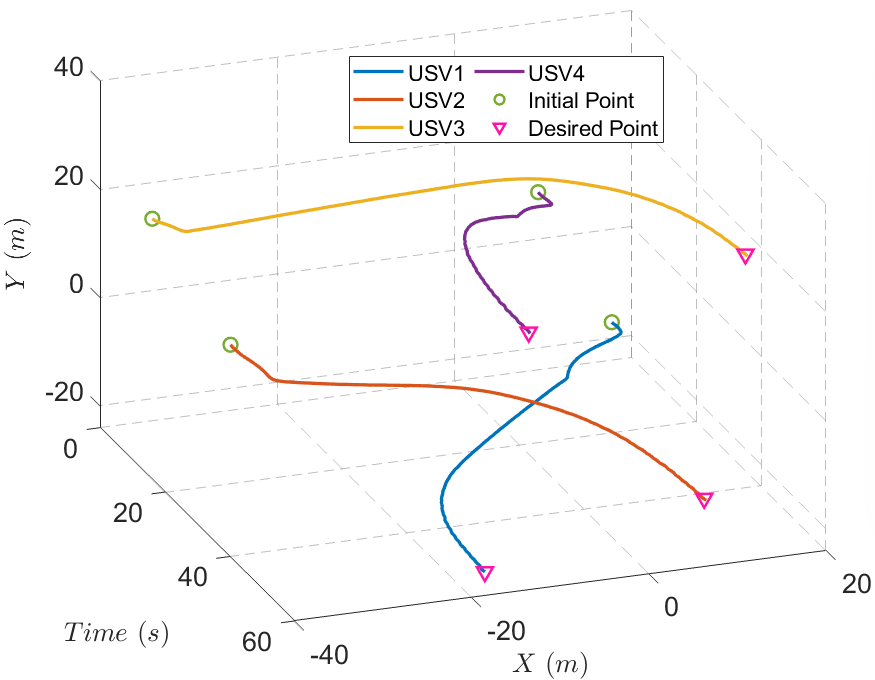}
	\caption{Planned trajectories of USVs from initial clustered positions to the desired quadrilateral formation in 3D space (X-Y-t).}
	\label{f6} 
\end{figure}
 \begin{figure}[!t]
	\centering
	\includegraphics[width=2.7in]{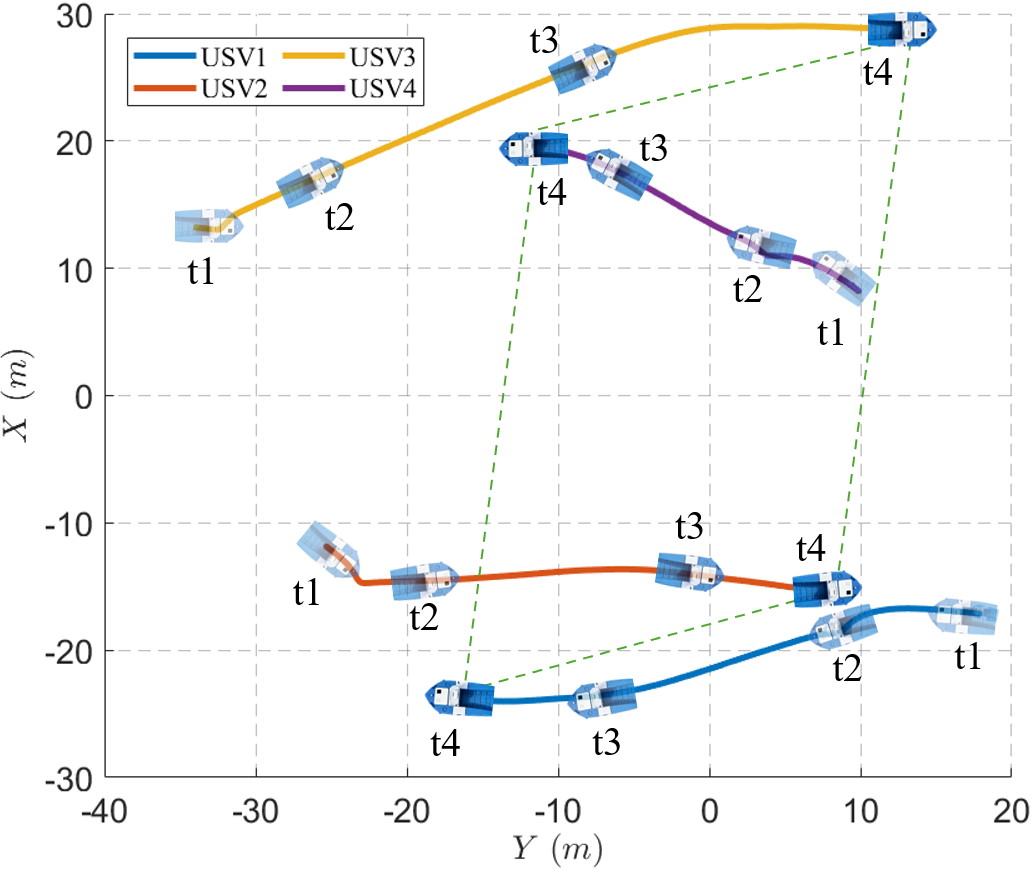}
	\caption{The actual trajectories of USVs, where $t_1=0s$, $t_2=10s$, $t_3=22s$, $t_4=54s$.}
	\label{f8} 
\end{figure}
\begin{figure}[htbp]
	\centering
	\begin{minipage}{0.4\textwidth}
		\centering
		\includegraphics[width=\textwidth]{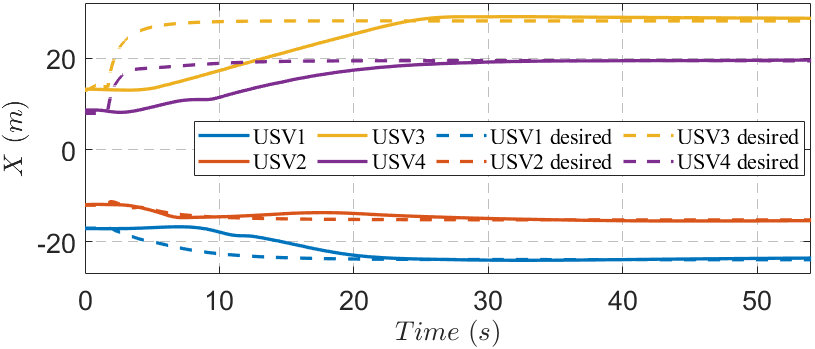}
	\end{minipage}
	\begin{minipage}{0.4\textwidth}
		\centering
		\includegraphics[width=\textwidth]{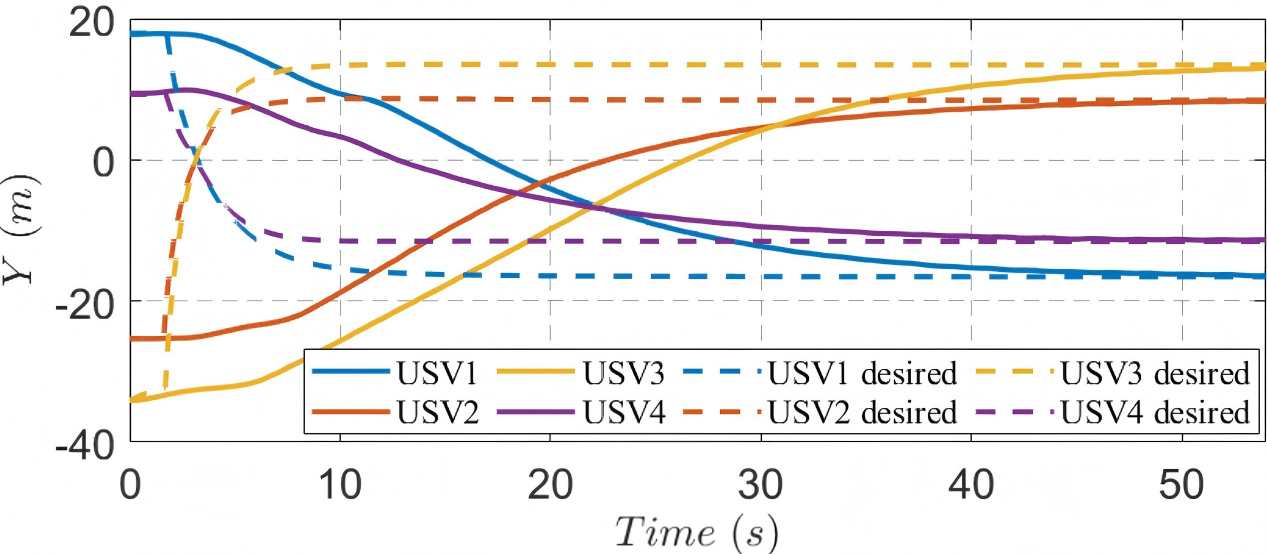}
	\end{minipage}
	\caption{Comparison of the planned (dashed) and actual (solid) trajectories.}
	\label{f7}
\end{figure}

As shown in Fig.~\ref{f6}-\ref{f7}, the proposed trajectory planning method is capable of generating smooth trajectories from the initial positions to the desired formation. Moreover, the designed controller can effectively track the target trajectories even in the presence of environmental disturbances. Next, we further evaluate the privacy preservation performance.

During the formation maneuvering of USVs, the exchanged information among them is confined to the centroid estimation $\hat{\eta}_i(t)$, which is the sole data transmitted over the network and susceptible to eavesdropping. Given that the estimation $\hat{\eta}_i(t)$ consistently differs from the actual position $Q_i(t)$, it is insufficient for an adversary to deduce the actual positions of the USVs. The comparison of $\hat\eta_i(t)=[\eta_{xi},\eta_{yi}]^\top$ and $Q_i(t)$ are shown in Fig.~\ref{f9}.
\begin{figure}[htbp]
	\centering
	\begin{minipage}{0.38\textwidth}
		\centering
		\includegraphics[width=\textwidth]{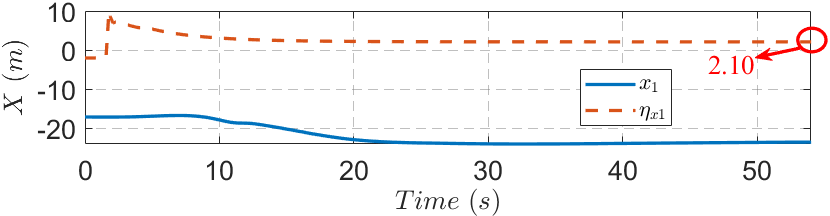}
	\end{minipage}
	\begin{minipage}{0.38\textwidth}
		\centering
		\includegraphics[width=\textwidth]{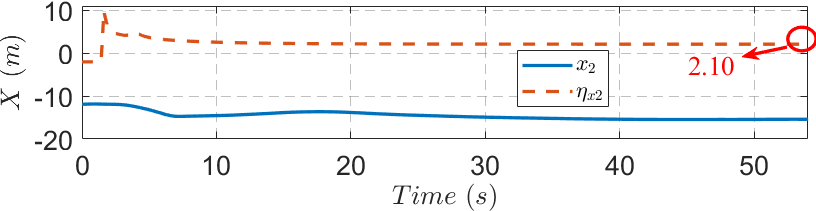}
	\end{minipage}
		\begin{minipage}{0.38\textwidth}
		\centering
		\includegraphics[width=\textwidth]{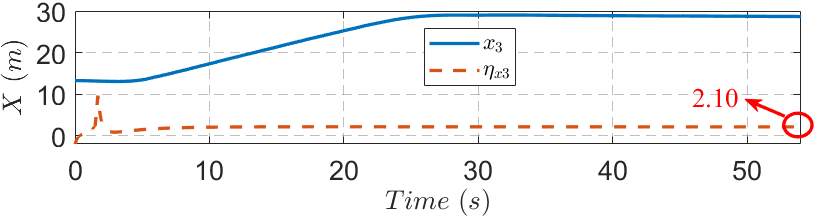}
	\end{minipage}
	\begin{minipage}{0.38\textwidth}
		\centering
		\includegraphics[width=\textwidth]{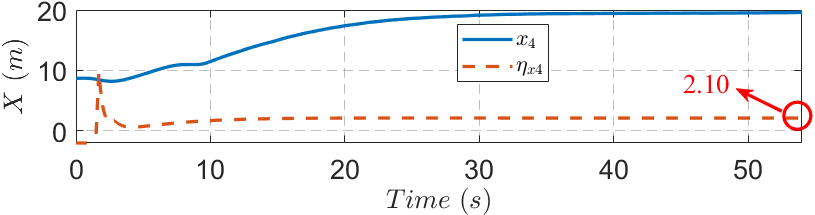}
	\end{minipage}
	\begin{minipage}{0.38\textwidth}
		\centering
		\includegraphics[width=\textwidth]{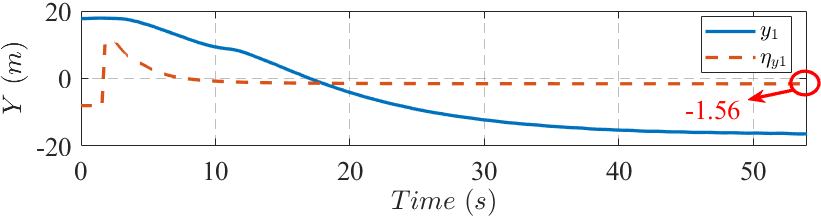}
	\end{minipage}
	\begin{minipage}{0.38\textwidth}
		\centering
		\includegraphics[width=\textwidth]{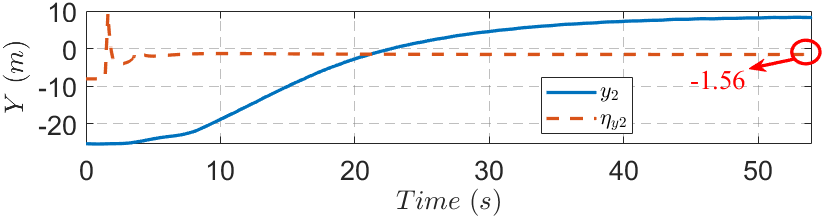}
	\end{minipage}
		\begin{minipage}{0.38\textwidth}
		\centering
		\includegraphics[width=\textwidth]{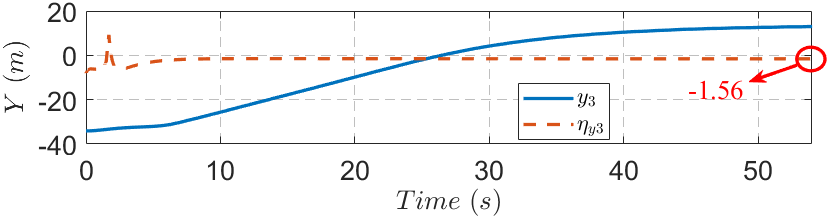}
	\end{minipage}
	\begin{minipage}{0.38\textwidth}
		\centering
		\includegraphics[width=\textwidth]{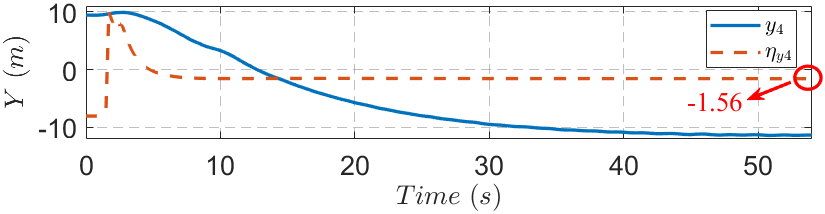}
	\end{minipage}
	\caption{Comparison of centroid estimation curves $\hat{\eta}_i(t)$ and actual trajectories $Q_i(t)$.}
	\label{f9}
\end{figure}

As illustrated in Fig.~\ref{f9}, the centroid estimation $\hat\eta_i(t)$ of each USV asymptotically converge to the same vector $[2.10,-1.56]^\top$, which coincides with the final geometric center of the fleet. It is also observed that the actual positions $Q_i(t)$ of the $i$-th USV remains significantly distinct from its estimation $\hat\eta_i(t)$ throughout the entire process, where $i=1,\cdots,4$. In conjunction with Theorem \ref{ppp1}, this indicates that neither HBC adversaries nor external eavesdroppers are able to infer the actual positions of neutral USVs from the exchanged estimation information. Consequently, the proposed hierarchical formation control scheme not only achieves the desired control objective but also ensures privacy preservation for neutral USVs against both types of adversaries.

\subsection{Extension to Time-Varying Formations}
The proposed formation control framework can be readily extended to accommodate time-varying formations. Specifically, by simply modifying the desired formation vector $\delta_i^*$ of each USV, the fleet can be reconfigured to achieve different geometric patterns without altering the underlying control law.

To demonstrate this capability, Fig.~\ref{f10} presents snapshots from the experiment, illustrating the USV fleet transitioning between two distinct parallelogram formation configurations. A complete video recording of the experiment is available at the following link: \textit{https://youtu.be/88P3ikO8CCg}. Interested readers are encouraged to view it for a more comprehensive visualization of the formation transition process.
 \begin{figure*}[!t]
	\centering
	\includegraphics[width=7in]{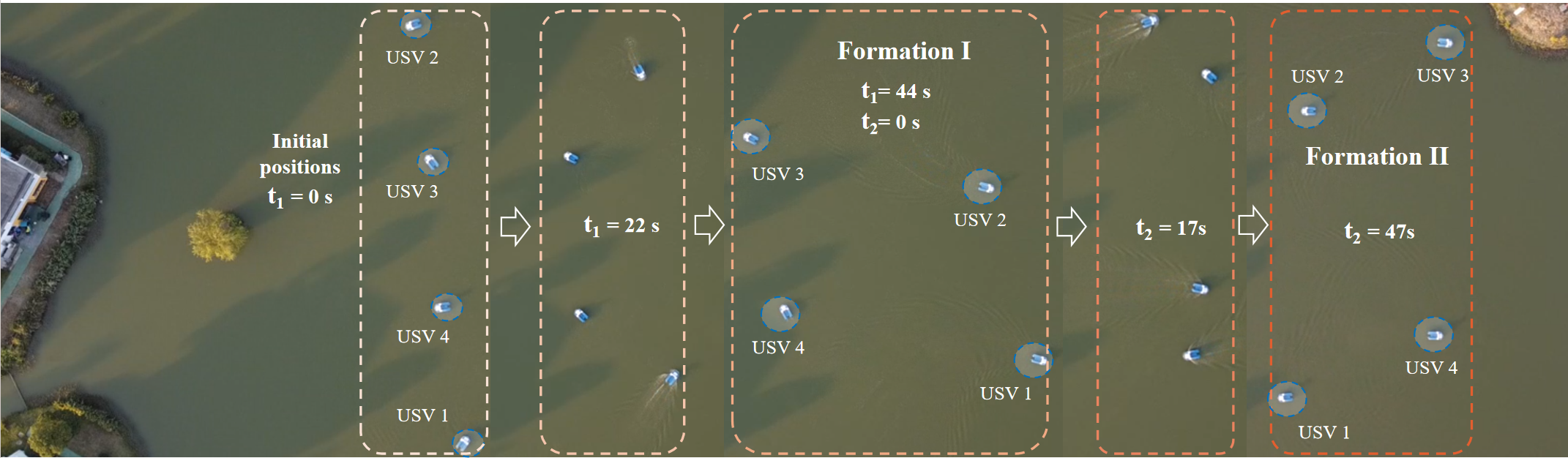}
	\caption{Snapshots of time-varying formation maneuvers.}
	\label{f10} 
\end{figure*}
\begin{remark}
It should be noted that the formation trajectories in the experiments were designed to ensure a safe distance between USVs at all times. As a result, no collision occurred during the trials, and collision avoidance mechanisms are beyond the scope of this paper. Further studies will address this issue in more complex scenarios.
\end{remark}
\section{Conclusion}
In this paper, a novel privacy-preserving formation control framework has been proposed for multi-USV systems within the PH framework. By utilizing the estimated centroid of the fleet as the sole interactive signal, the scheme effectively protects the privacy of each USV without compromising formation performance. To seamlessly integrate  the trajectory planning and physical execution, a passivity-based tracking controller has been developed, achieving high-precision formation tracking while robustly safeguarding privacy against HBC adversaries and external  eavesdroppers. The rigorous stability analysis and field experiments on a practical USV platform confirm that this paper provides a feasible and secure solution for the distributed formation of USVs in privacy-sensitive environments. In the future, we aim to extend the proposed framework to complex environments with dynamic obstacles and intermittent communication constraints.




\end{document}